\documentclass[a4paper,10pt,english,leqno]{amsart}
\def\ocirc#1{\ifmmode\setbox0=\hbox{$#1$}\dimen0=\ht0
    \advance\dimen0 by1pt\rlap{\hbox to\wd0{\hss\raise\dimen0
    \hbox{\hskip.2em$\scriptscriptstyle\circ$}\hss}}#1\else
    {\accent"17 #1}\fi}
\usepackage{color}
\usepackage{dsfont}
\usepackage[colorlinks = false]{hyperref}
\usepackage{mathrsfs}
\usepackage{mathtools}
\usepackage{amssymb}
\usepackage{bbm}
\numberwithin{equation}{section}
\allowdisplaybreaks[1]
\usepackage{soul}
\usepackage{tikz}
\usetikzlibrary{decorations.pathreplacing}

\usepackage{environ}
\makeatletter
\NewEnviron{Lalign}{\tagsleft@true\begin{align}\BODY\end{align}}
\makeatother

\usepackage[shortcuts,cyremdash]{extdash}
\usepackage{enumerate}
\newtheorem{theorem}{Theorem}[section]
\newtheorem{lemma}[theorem]{Lemma}

\newtheorem{definition}[theorem]{Definition}
\newtheorem{remark}[theorem]{Remark}

\newcommand{\R}{\mathbb{R}}
\newcommand{\Z}{\mathbb{Z}}

\newcommand{\N}{\mathbb{N}}

\newcommand{\dt}{\,dt}

\renewcommand{\dt}{\partial_t}
\renewcommand{\O}{\Omega}
\newcommand{\Oi}{\Omega_i^{\varepsilon}}
\newcommand{\Oe}{\Omega_e^{\varepsilon}}
\newcommand{\Oj}{\Omega_j^{\varepsilon}}
\newcommand{\Ge}{\Gamma^{\varepsilon}}
\newcommand{\e}{\varepsilon}
\newcommand{\p}{\varphi}

\newcommand{\G}{\Gamma}

\newcommand{\Tau}{\mathcal{T}}

\newcommand{\weakstar}{\overset{\star}\rightharpoonup}
\newcommand{\weak}{\rightharpoonup}

\newcommand{\seq}[1]{\left\{#1\right\}}
\newcommand{\abs}[1]{\left|#1\right|}

\newcommand{\norm}[1]{\left\|#1\right\|}

\DeclareMathOperator*{\Div}{div}

\title[Two-scale homogenization of the bidomain model]
{The cardiac bidomain model and homogenization}

\author[E. Grandelius]{E. Grandelius}
\address[Erik Grandelius]
{\newline Department of mathematics
\newline University of Oslo
\newline P.O. Box 1053, Blindern
\newline N--0316 Oslo, Norway} 
\email[]{erikgran@math.uio.no}

\author[K. H. Karlsen]{K. H. Karlsen}
\address[Kenneth Hvistendahl Karlsen]
{\newline Department of mathematics
\newline University of Oslo
\newline P.O. Box 1053,  Blindern
\newline N--0316 Oslo, Norway} 
\email[]{kennethkarlsen@me.com}



\thanks{This work was supported by the Research Council 
of Norway (project 250674/F20).}

\date{\today}

\begin{document}

\begin{abstract} 
We provide a rather simple proof of a homogenization 
result for the bidomain model of cardiac electrophysiology. 
Departing from a microscopic cellular model, we 
apply the theory of two-scale convergence 
to derive the bidomain model. To allow for some relevant 
nonlinear membrane models, we make essential 
use of the boundary unfolding operator.
There are several complications preventing 
the application of standard homogenization results, including the 
degenerate temporal structure of the bidomain equations 
and a nonlinear dynamic boundary condition 
on an oscillating surface. 
\end{abstract}

\maketitle

\tableofcontents

\section{Introduction}\label{sec:Introduction}

The bidomain model \cite{Tung78,ColliBook,Sundnes} is widely used 
as a quantitative description of the electric activity in cardiac tissue. 
The relevant unknowns are the intracellular ($u_i$) and extracellular ($u_e$) 
potentials, along with the so-called transmembrane potential ($v:=u_i-u_e$). 
In this model, the intra- and extracellular spaces are 
considered as two separate homogeneous domains superimposed 
on the cardiac domain. The two domains are separated by 
the cell membrane creating a discontinuity surface for 
the cardiac potential. Conduction of electrical signals in cardiac tissue relies 
on the flow of ions through channels in the cell membrane. 
In the bidomain model, the celebrated Hodgkin-Huxley \cite{Hodgkin} 
framework is used to dynamically couple the intra- and extracellular potentials 
through voltage gated ionic channels. 

The bidomain model can be viewed as a PDE system consisting of two 
degenerate reaction-diffusion equations involving the unknowns $u_i,u_e,v$ and 
two conductivity tensors $\sigma_i,\sigma_e$. These equations are supplemented 
by a nonlinear ODE system for the dynamics of the ion channels. 
The bidomain model is often derived heuristically by interpreting 
$\sigma_i,\sigma_e$ as some sort of ``average" conductivities,  applying 
Ohm's electrical conduction law and the continuity equation (conservation 
of electrical charge) to the intracellular and extracellular 
domains \cite{ColliBook,Sundnes}.

Starting from a more accurate microscopic (cell-level) model of 
cardiac tissue, with the heterogeneity of the underlying cellular 
geometry represented in great detail, it is possible to heuristically 
derive the bidomain model (tissue-level) using the 
multiple scales method of homogenization. 
This derivation was first carried out in \cite{Neu:1993aa}. 
It should be noted that the microscopic model is in 
general too complex to allow for full organ simulations, although 
there have been some very recent efforts in that direction \cite{Tveito17}. 
The complexity of cell-level models, which themselves 
can be heuristically derived from the Poisson-Nernst-Planck equations 
\cite{Richardson:2009aa}, motivates the search for 
simpler homogenized (macroscopic) models. 
The work \cite{Neu:1993aa} assumes, as we do herein, that 
cardiac tissue can be viewed as a uniformly oriented periodic 
assembly of cells (see also \cite{Colli,Henriquez}). 
There have been some attempts to remove this 
assumption. We refer to \cite{Keener:1998ab,Keener:1996aa,Richardson:2011aa} 
for extensions to  somewhat more realistic tissue geometries.

Despite the widespread use of the  bidomain model, there are 
few mathematical rigorous derivations of the model from a microscopic 
description of cardiac tissue. From a mathematical point of view, 
rigorous homogenization is often linked to the study of the 
asymptotic behavior (convergence) of solutions to PDEs with oscillating 
coefficients. In the literature several approaches have 
been developed to handle this type of problem, like Tartar's method 
of oscillating test functions, $\Gamma$-convergence, 
two-scale convergence, and the unfolding method. We refer to \cite{Donato} 
for an accessible introduction to the mathematics of homogenization 
and for an overview of the different homogenization methods. 

We are aware of two earlier works \cite{Amar2013,Pennacchio2005} 
containing rigorous homogenization results for the 
bidomain model (but see \cite{Donato:2015aa,Donato:2011aa,Yang:2014aa}
for examples of elliptic and parabolic equations on``two-component" domains). 
With a fairly advanced proof involving $\Gamma$-convergence, the 
De Giorgi ``minimizing movement" approach, time-discretization, 
variational problems, and two limit procedures, the homogenization 
result in \cite{Pennacchio2005} covers the generalized 
FitzHugh-Nagumo ionic model \cite{FitzHugh1955}. 
The proof of the result in \cite{Amar2013} is more basic in the 
sense that it employs only two-scale convergence arguments, but 
it handles only a restricted class of ionic models.  

We mention that there are several complications preventing 
the application of standard homogenization results 
(for elliptic/parabolic equations) to the bidomain equations, including 
its degenerate structure (seen at the tissue-level), 
resulting from differing anisotropies of the intra- 
and extracellular spaces, 
and the highly nonlinear, oscillating dynamic 
boundary condition (seen at the cell-level).

The main contribution of our paper is to provide a simple homogenization 
proof that can handle some relevant nonlinear 
membrane models (the generalized FitzHugh-Nagumo model), 
relying only on basic two-scale convergence techniques. 
We now explain our contribution in more detail. 
The point of departure is the following microscopic 
model \cite{ColliBook,Colli,Henriquez,Veneroni} for the 
electric activity in cardiac tissue:
\begin{equation}\label{micro}
	\begin{split}
		& -\Div \left( \sigma_i^\e \nabla u_i^{\e}\right) = s_i^{\e}  
		\quad  \mbox{in} \; (0,T)\times \Oi,
		\\ & 
		-\Div \left( \sigma_e^\e \nabla u_e^{\e}\right) = s_e^{\e}  
		\quad \mbox{in} \; (0,T)\times \Oe,
		\\ & 
		\e \left( \partial_t {v^{\e}} + I(v^{\e},w^{\e})\right)
		=-\nu \cdot \sigma_{i}^{\e} \nabla u_i^{\e}
		\quad \mbox{on} \; (0,T)\times \Ge,
		\\ & 
		\e \left( \partial_t {v^{\e}} + I(v^{\e},w^{\e})\right)
		=-\nu \cdot \sigma_{e}^{\e} \nabla u_e^{\e}
		\quad \mbox{on} \; (0,T)\times \Ge,
		\\ & 
		\dt w^{\e} = H(v^{\e},w^{\e}) 
		\quad \text{on}  \; (0,T)\times\Ge,
	\end{split}	
\end{equation}
where $\nu$ denotes the unit normal pointing out 
of $\Oi$ (and into $\Oe$). Cardiac tissue consists 
of an assembly of elongated cylindrical-shaped cells 
coupled together (end-to-end and side-to-side) to 
provide intercellular communication. The entire cardiac 
domain $\Omega\subset \R^3$ is viewed as a ``two-component" domain 
and split into two $\e$-periodic open sets $\Oi$, $\Oe$ 
corresponding to the intra- and extracellular spaces. 
The sets $\Oi,\Oe$, which are assumed to be disjoint and connected, are 
separated by an $\e$-periodic surface $\Ge$ representing the cell membrane, 
so that $\Omega=\Oi\cup \Oe\cup \Ge$. The main geometrical assumption 
is that the intra- and extracellular domains are $\e$-dilations of some reference 
cells $Y_i,Y_e\subset Y:=[0,1]^3$, periodically 
repeated over $\R^3$. Although our results are valid 
for general Lipschitz domains, for simplicity of presentation, we 
assume that the entire cardiac domain $\O$ is the open cube
\begin{equation}\label{def:domain}
	\O=(0,1)\times(0,1)\times (0,1).
\end{equation}

In \eqref{micro}, $\sigma_j^{\e}$ is the conductivity tensor 
and $s_j^{\e}$ is the stimulation current, relative to $\Oj$ for $j=i,e$.
The functions $s_i^\e,s_e^\e$ are assumed to be at least bounded in $L^2$, 
independently of $\e$. As usual in homogenization theory, the
conductivity tensors $\sigma_i^{\e}, \sigma_e^{\e}$ are assumed to have the form 
$$
\sigma_j^\e(x)=\sigma_j\left(x,\frac{x}{\e}\right), \qquad j=i,e,
$$
where $\sigma_j=\sigma_j(x,y)$ satisfies the usual conditions 
of uniform ellipticity and periodicity (in $y$).  
Despite the fact that the inhomogeneities of the domains impose 
$\e$-oscillations in the conductivity tensors (via gap junctions), 
the main source of inhomogeneity in the microscopic 
model is not the conductivities $\sigma_i^\e$ and $\sigma_e^\e$, but 
the domains $\Oi$ and $\Oe$ themselves. We allow for 
inhomogeneous and oscillating conductivities for the sake of generality.

We denote by $u_j^{\e}$ the electric potential in $\Oj$ ($j=i,e$).
On $\Ge$, $v^{\e} := u_i^{\e}-u_e^{\e}$ is the transmembrane 
potential and $I(v^{\e},w^{\e})$ is the ionic current depending 
on $v^{\e}$ and a gating variable $w^{\e}$. 
The left-hand side of the third and fourth equations in \eqref{micro}  
describes the current across the membrane as having a capacitive 
component, depending on the time derivative of the transmembrane potential, 
and a nonlinear ionic component $I$ corresponding to the 
chosen membrane model. In this article we consider 
the generalized FitzHugh-Nagumo model \cite{FitzHugh1955}. 
We choose to focus on this membrane model for definiteness, but our 
arguments can be adapted to many other models satisfying 
reasonable technical assumptions 
\cite{Boulakia2008,Bourgault,ColliBook,Sundnes,Veneroni,Veneroni:2009aa}.

For each fixed $\e>0$, the functions $\sigma_j^{\e}, s_j^{\e}, I,H$ in \eqref{micro} 
are given and we wish to solve for $(u_i^\e,u_e^\e,v^\e,w^\e)$. 
To this end, we must augment the system \eqref{micro} 
with initial conditions for $v^\e, w^\e$ and Neumann-type boundary 
conditions for $u_i^\e, u_e^\e$ (ensuring 
no current flow out of the heart):
\begin{equation}\label{eq:ib-cond}
	\begin{split}
		&v^{\e}|_{t=0}=v_0^\e \;\;  \text{in $\O$}, \quad
		w^{\e}|_{t=0}=w_0^\e\;\;  \text{in $\O$}, 
		\\ &
		n \cdot\sigma_j  \nabla u_j^{\e}  = 0  
		\;\; 
		\text{on $(0,T)\times \left(\partial \O 
		\cap \partial \Oj\right)$}, \; j=i,e,
	\end{split}
\end{equation}
where $n$ is the outward unit normal to $\O$. 
It is proved in \cite{Colli,Veneroni} that the microscopic 
bidomain model \eqref{micro},  \eqref{eq:ib-cond} possesses 
a unique weak solution. This solution satisfies a series of a priori estimates.
For us it is essential to know how these estimates 
depend on the parameter $\e$. We will therefore outline 
a proof of these estimates.

The dimensionless number $\e$ is a small positive number 
representing the ratio of the microscopic and macroscopic 
scales, that is, considering $\O$ as fixed, it is proportional to 
the cell diameter. The goal of homogenization is to investigate 
the limit of a sequence of solutions 
$\seq{\left(u^{\e}_i,u_e^{\e},v^{\e},w^{\e}\right)}_{\e>0}$ 
to \eqref{micro}, \eqref{eq:ib-cond}. By the 
multiple scales method \cite{Donato,Colli,Henriquez}, 
the electric potentials $u^{\e}_i,u_e^{\e},v^{\e}$ and 
the state variable $w^\e$ exhibit the following asymptotic 
expansions in powers of the parameter $\e$:
\begin{align*}
	u_j^\e(t,x,y) & =u_j(t,x,y) + \e u_j^{(1)}(t,x,y)
	+\e^2 u_j^{(2)}(t,x,y)
	+\cdots \quad (j=i,e),
	\\ 
	v^\e(t,x,y) & =v(t,x,y) + \e v^{(1)}(t,x,y)
	+\e^2 v^{(2)}(t,x,y)
	+\cdots,
	\\  
	w^\e(t,x,y) & =w(t,x,y) + \e w^{(1)}(t,x,y)
	+\e^2 w^{(2)}(t,x,y)
	+\cdots,
\end{align*}
where $y= x/\e$ denotes the microscopic variable, and each term in 
the expansions is a function of both the slow (macroscopic) 
variable $x$ and the fast (microscopic) variable $y$, periodic 
in $y$. Substituting the above expansions into \eqref{micro}, 
and equating all terms of the same orders in powers of $\e$, 
we obtain after some routine arguments that 
the zero order terms $u_i,u_e,v,w$ are independent of the 
fast variable $y$ and satisfy the (macroscopic) 
bidomain model \cite{Colli,Henriquez,Pennacchio2005}
\begin{equation}\label{macro}
	\begin{cases}
		|\G|\dt v -\Div \left( M_i \nabla u_i \right)+ |\G|I(v,w) 
		= |Y_i|s_{i}, & \quad \text{in $(0,T)\times \O$}, \\
		|\G|\dt v + \Div  \left( M_e \nabla u_e \right)  + |\G|I(v,w) 
		= - |Y_e|s_{e}, & \quad \text{in $(0,T)\times \O$},\\ 
		\dt w = H(v,w), & \quad \text{in $(0,T)\times \O$},
	\end{cases}
\end{equation}
where the homogenized conductivity tensors 
$M_i(x), M_e(x)$ are given by 
\begin{equation}\label{Mj}
	M_j(x)=\int_{Y_j} \sigma_j(x,y)\left( I
	+ \nabla_y \chi_j(x,y) \right) \, dy, 
	\qquad j=i,e,
\end{equation}
and the $y$-periodic (vector-valued) 
function $\chi_j=\chi_j(x,y)$ solves the cell problem
\begin{equation}\label{chi}
	\begin{cases}
		-\Div_y \left( \sigma_j \nabla_y \chi_j \right) 
		= -\Div_y \sigma_j, \quad &\mbox{in } \O \times Y_j, \\
		\nu \cdot \sigma_j \nabla_y \chi_j = \nu \cdot \sigma_j,  
		\quad &\mbox{on } \O\times \Gamma.
	\end{cases}
\end{equation}
Note that the effective potentials $u_i, u_e$ in \eqref{macro} are 
defined at every point of $\O$, while in the microscopic model they live 
on disjoint sets $\Oi, \Oe$. In \eqref{macro}, \eqref{Mj}, \eqref{chi} 
the sets $Y_i,Y_e$ are the intra and extracellular 
spaces within the reference unit cell $Y$, separated by the cell 
membrane $\G$ (see Section \ref{sec:preliminaries} for details). 
It is worth noting that the bidomain model is often stated in terms of 
the ``geometric" parameter $\chi=\frac{\abs{\Gamma}}{\abs{Y}}=\abs{\Gamma}$ 
representing the surface-to-volume ratio of the cardiac cells. 

Regarding the existence and uniqueness of properly 
defined solutions, standard theory for parabolic-elliptic 
systems does not apply naturally to the bidomain model \eqref{macro}.
A number of works \cite{Karlsen,Boulakia2008,Bourgault,Colli,Veneroni:2009aa} 
have recently provided well-posedness results for \eqref{macro}, applying differing 
solution concepts and technical frameworks. 

As alluded to earlier, we will provide a rigorous derivation of the homogenized 
system \eqref{macro}, \eqref{Mj}, \eqref{chi} based on the theory of two-scale 
convergence (see \cite{Nguetseng} and \cite{Allaire,Allaire2,Lukkassen:2002aa}).
This result is not covered by standard parabolic homogenization theory. 
A complication is the nonlinear dynamic boundary condition (posed on an underlying 
oscillating surface), which makes it difficult to pass to the limit in 
\eqref{micro} as $\e\to 0$.  The aim is to prove that a sequence  
$\seq{u_i^{\e},u_e^{\e},v^{\e},w^{\e}}_{\e>0}$ of solutions to the 
microscopic problem two-scale converges to the solution of 
the bidomain model \eqref{macro}. However, two-scale convergence is 
not ``strong enough" to justify passing to the limit in the 
nonlinear boundary condition. To handle this difficulty we 
use the boundary unfolding operator \cite{Cioranescu}, establishing 
strong convergence of $\Tau_{\e}^b(v^{\e})$ in $L^2((0,T)\times\O\times \G)$, 
where $\Tau_{\e}^b$ denotes the boundary unfolding operator. 
The boundary unfolding operator makes our proof flexible 
enough to handle a range of membrane models, 
exemplified by the generalized FitzHugh-Nagumo model.

Unfolding operators, presented 
and carefully analyzed in \cite{Cioranescu:2008aa,Cioranescu}, 
facilitate elementary proofs of classical homogenization 
results on fixed as well as perforated domains/surfaces. 
An unfolding operator $\Tau_\e$ maps a function $v(x)$ defined on 
an oscillating domain/surface to a higher dimensional function $\Tau_{\e}(v)(x,y)$ 
on a fixed domain, to which one can apply standard convergence 
theorems in fixed $L^p$ spaces. Reflecting the ``two-component" nature of 
the cardiac domain, it makes sense to use two 
unfolding operators $\Tau_\e^{i,e}$, linked to the intra- and extracellular 
domains $\O^\e_{i,e}$. In this paper, however, we mainly unfold functions 
defined on the cell membrane, utilizing the boundary unfolding operator $\Tau_\e^b$. 

For somewhat similar unfolding of ``two-component" domains separated by a 
periodic boundary, see \cite{Donato:2015aa,Donato:2011aa,Yang:2014aa}. 
For other relevant works that combine two-scale convergence and unfolding 
methods, we refer to \cite{Neuss,Gahn:2016aa,Graf:2014aa,Marciniak-Czochra:2008aa,Neuss-Radu:2007aa}. 
Among these, our work borrows ideas mostly 
from \cite{Neuss,Gahn:2016aa,Neuss-Radu:2007aa}. 
      
The remaining part of the paper is organized as follows: 
In Section \ref{sec:preliminaries}, we collect relevant 
functional spaces and analysis results. Moreover, we gather 
definitions and tools linked to two-scale convergence and unfolding operators. 
In Section \ref{sec:wellposedness}, we define precisely what is meant by a 
weak solution of the microscopic problem \eqref{micro}, state a 
well-posedness result, and establish several ``$\e$-independent" a priori estimates.  
The main homogenization result is stated and proved in Section \ref{sec:convergence}.

\section{Preliminaries}\label{sec:preliminaries}

\subsection{Some functional spaces and tools}\label{sec:notation}
For a general review of integer and fractional order 
Sobolev spaces (on Lipschitz domains) and relevant analysis tools, 
see \cite[Chaps.~2 \& 3]{Boyer:2012aa}
and \cite[Chap.~3]{McLean:2000aa}. For relevant background material 
on mathematical homogenization, we refer to \cite{Donato}. 

Let $\O \subset \R^3$ be a bounded open set with 
Lipschitz boundary. We denote by $C_0^{\infty}(\O)$ the 
(infinitely) smooth functions with compact support in $\O$. 
The space of smooth $Y$-periodic functions is 
denoted by $C^{\infty}_{\mathrm{per}}(Y)$. The closure of this 
space under the norm $\|\nabla(\cdot)\|_{L^2(Y)}$ 
is denoted by $H_{\mathrm{per}}^1(Y)$. We write $H^s$ for the 
$L^2$-based Sobolev spaces $W^{s,2}$ ($s\in (0,1]$). 

We make use of Sobolev spaces on surfaces, as defined 
for example in \cite[p.~34]{Lions:1972aa} 
and \cite[p.~96]{McLean:2000aa}. Specifically, we use the 
(Hilbert) space $H^{1/2}(\G)$, for a two-dimensional 
Lipschitz surface $\G \subset \O$, equipped with the norm 
\begin{equation*}
	\norm{u}^2_{H^{1/2}(\G)}
	=\norm{u}_{L^2(\G)}^2+\abs{u}_{H_0^{1/2}(\G)}^2,
\end{equation*}
where 
$$
\abs{u}_{H_0^{1/2}}^2=\int_{\G}\int_{\G}
\frac{|u(x)-u(x')|^2}{|x-x'|^3}\, dS(x)\, dS(x'),
$$
and $dS$ is the two-dimensional surface measure. We define 
the dual space of $H^{1/2}(\G)$ as $H^{-1/2}(\G):=(H^{1/2}(\G))^*$, equipped 
with the norm of dual spaces 
$$
\norm{u}_{H^{-1/2}(\G)}
:=\underset{\norm{\phi}_{H^{1/2}(\G)}=1}{\sup_{\phi\in H^{1/2}(\G)}}
\left \langle 
u, \phi 
\right\rangle_{H^{-1/2}(\G),H^{1/2}(\G)}.
$$

The following trace inequality holds:
\begin{equation}\label{trace1}
	\norm{u|_{\G}}_{H^{1/2}(\G)} 
	\leq C \norm{u}_{H^1(\O)}, \quad u \in H^1(\O).
\end{equation} 
Any function in $H^{1/2}(\G)$ can be characterized as 
the trace of a function in $H^1(\O)$. 
The trace map has a continuous right inverse 
$\mathcal{I}:H^{1/2}(\G)\rightarrow H^1(\O)$, satisfying
\begin{equation}\label{inverseTrace}
	\norm{\mathcal{I}(u)}_{H^1(\O)}
	\leq C\norm{u}_{H^{1/2}(\G)}, 
	\quad \forall u \in H^{1/2}(\G),
\end{equation}
where the constant $C$ depends only on $\G$. 
We need the Sobolev inequality
\begin{equation}\label{SobolevInequality}
	\norm{u}_{L^4(\G)} \leq C\norm{u}_{H^{1/2}(\G)}.
\end{equation}
Indeed,  $H^{1/2}(\G)$ is continuously embedded in $L^p(\G)$ 
for $p\in [1,4]$. This embedding is compact for $p\in [1,4)$. 
In particular, $H^{1/2}(\G)$ is compactly embedded in $L^2(\G)$.

Let $X$ be a separable Banach space $X$ and $p\in [1,\infty]$, 
We make routinely use of Lebesgue-Bochner spaces such as
$L^p(\Omega;X)$ and $L^p(0,T;X)$. We also use the spaces of 
continuous functions from $\O$ to $X$ and $(0,T)$ to $X$, 
denoted by $C(\O;X)$ and $C(0,T;X)$, respectively, 
and the similar spaces with $C$ replaced by $C^p$ or $C^p_0$.
If $X$ is a Banach space, then $X/\R$ denotes 
the space consisting of classes of functions in $X$ that are 
equal up to an additive constant.

Recall that $H^{1/2}(\G)$ is a Hilbert space 
embedded in a continuous and dense way in $L^2(\G)$. 
The (Lions-Magenes) integration-by-parts formula holds 
for functions $u_1,u_2$ that belong to the Banach space
\begin{align*}
	\mathcal{V}_{\G,T}
	& = \Bigl\{
	u \in L^2(0,T; H^{1/2}(\G))\cap L^4((0,T)\times \G) \, \big| \, 
	\\ & \qquad\qquad 
	\dt u \in L^2(0,T;H^{-1/2}(\G))+L^{4/3}((0,T)\times \G) 
	\Bigr\},
\end{align*}
equipped with the norm
$$
\norm{u}_{\mathcal{V}_{\G,T}}
=\norm{u}_{L^2(0,T;H^{1/2}(\G))\cap L^4((0,T)\times \G)}
+\norm{\dt u}_{L^2(0,T;H^{-1/2}(\G))+L^{4/3}((0,T)\times \G)},
$$
where $\norm{u}_{X_1\cap X_2}=
\max \left(\norm{u}_{X_1},\norm{u}_{X_2} \right)$, 
$\norm{u}_{X_1+X_2}=\inf\limits_{u=u_1+u_2} \left(\norm{u_1}_{X_1}
+\norm{u_2}_{X_2}\right)$. Indeed, for $u_1,u_2 \in \mathcal{V}_{\G,T}$, 
$[0,T]\ni t\mapsto \left ( u_1(t), u_2(t) 
\right)_{L^2(\G)}$ is continuous and 
\begin{equation}\label{eq:int-by-parts-general}
	\begin{split}
		& \int_{t_1}^{t_2} 
		\left \langle 
		\dt u_1, u_2 
		\right\rangle\, dt 
		+\int_{t_1}^{t_2} 
		\left \langle 
		\dt u_2,u_1  
		\right\rangle\, dt
		\\ & \qquad
		= \left ( 
		u_1(t_2), u_2(t_2) 
		\right)_{L^2(\G)}
		-\left ( 
		u_1(t_1), u_2(t_1) 
		\right)_{L^2(\G)},
	\end{split}
\end{equation}
for all $t_1,t_2\in [0,T]$, $t_1<t_2$, where 
$\left\langle \cdot, \cdot \right \rangle$ denotes 
the duality pairing between $H^{-1/2}(\G)+L^{4/3}(\G)$ 
and $H^{1/2}(\G)\cap L^4(\G)$. For a proof of \eqref{eq:int-by-parts-general} 
that can be adapted to our situation, see e.g.~\cite[p.~99]{Boyer:2012aa}. 

Taking $u_1=u_2=u\in \mathcal{V}_{\G,T}$, we obtain the chain rule
\begin{equation}\label{eq:chain-rule}
	\begin{split}
		\int_{t_1}^{t_2} 
		\left \langle 
		\dt u, u 
		\right \rangle \, dt 
		= \frac12 \norm{u(t_2)}_{L^2(\G)}^2 
		-\frac12 \norm{u(u_1)}_{L^2(\G)}^2.
	\end{split}
\end{equation}

Adapting standard arguments (see e.g.~\cite[p.~101]{Boyer:2012aa}), 
the embedding
\begin{equation}\label{cont-embeding}
	\mathcal{V}_{\G,T}\hookrightarrow C(0,T; L^2(\G))
\end{equation}
is continuous. Indeed, this result follows from the continuity of the 
squared norm $t\mapsto \norm{u(t)}_{L^2(\G)}^2$ (see above) and  
the weak continuity of $u$ in $L^2(\G)$. The latter results 
from an easily obtained bound on $u$ in $L^\infty(0,T;L^2(\G))$ 
and the continuity of $u$ in $H^{-1/2}(\G)$, both facts being 
deducible from \eqref{eq:int-by-parts-general}. 

Let us dwell a bit further on the time continuity 
of functions in $\mathcal{V}_{\G,T}$. 
By \eqref{SobolevInequality}, $H^{1/2}\subset L^4(\G)$ 
and so $L^{4/3}(\G)\subset H^{-1/2}(\G)$. Therefore, 
\begin{equation*}
	L^2(0,T;H^{-1/2}(\G))+L^{4/3}((0,T)\times \G) 
	\subset L^{4/3}(0,T;H^{-1/2}(\G)).
\end{equation*}
With $u_1=u\in \mathcal{V}_{\G,T}$ and $u_2=\phi\in H^{1/2}(\G)$ 
in \eqref{eq:int-by-parts-general}, it follows that
\begin{align*}
	\left ( 
	u(t_2)-u(t_1), \phi 
	\right)_{L^2(\G)}
	&
	=\int_{t_1}^{t_2} 
	\left \langle 
	\dt u, \phi 
	\right \rangle \, dt
	\\ & 
	\le \norm{\phi}_{H^{1/2}(\G)}
	\int_{t_1}^{t_2} \norm{\dt u}_{H^{-1/2}(\G)}\, dt
	\\ & \le 
	\norm{\phi}_{H^{1/2}(\G)} 
	\norm{\dt u}_{L^{4/3}(t_1,t_2;H^{-1/2}(\G))}
	(t_2-t_1)^{1/4}.
\end{align*}
Fix a small shift $\Delta_t>0$. Specifying 
$t_1=t\in (0,T-\Delta_t)$, $t_2=t+\Delta_t$, 
and $\phi=u(t+\Delta_t,\cdot)-u(t,\cdot)$ gives 
\begin{align*}
	&\norm{u(t+\Delta_t,\cdot)-u(t,\cdot)}_{L^2(\G)}^2
	\\ & \qquad
	\le \norm{u(t+\Delta_t,\cdot)-u(t,\cdot)}_{H^{1/2}(\G)}
	\norm{\dt u}_{L^{4/3}(0,T;H^{-1/2}(\G))}
	\Delta_t^{1/4}.
\end{align*}
Integrating this inequality over $t\in (0,T-\Delta_t)$, 
accompanied by a few elementary manipulations, results in the 
temporal translation estimate
\begin{equation*}
	\begin{split}
		& \norm{u(\cdot+\Delta_t,\cdot)
		-u(\cdot,\cdot)}_{L^2(0,T-\Delta_t;L^2(\G))}
		\\ & \quad
		\le C_T\norm{u}_{L^2(0,T;H^{1/2}(\G))}^{1/2} 
		\norm{\dt u}_{L^{4/3}(0,T;H^{-1/2}(\G))}^{1/2}
		\Delta_t^{1/8}, 
		\quad u\in \mathcal{V}_{\G,T},
	\end{split}
\end{equation*}
where $C_T=2^{1/2}T^{1/4}$. 
A similar estimate holds for negative $\Delta_t$.

There is a compact embedding of $\mathcal{V}_{\G,T}$ in 
$L^2(0,T;L^2(\G))$. As pointed out above, 
$\mathcal{V}_{\G,T}$ is a subset of 
$\seq{u\in L^2(0,T;H^{1/2}(\G):\dt u \in L^{4/3}(0,T;H^{-1/2}(\G))}$, 
which is compactly embedded in $L^2(0,T;L^2(\G))$ by the Aubin-Lions theorem.

We need a generalization of this 
result due to Simon \cite{Simon:1987vn}. 
Given two Banach spaces $X_1\subset X_0$, with 
$X_1$ compactly embedded in $X_0$, let $\mathcal{K}$ be a collection 
of functions in $L^p(0,T;X_0)$, $p\in [1,\infty]$. 
The work \cite{Simon:1987vn} supplies several results ensuring the 
compactness of $\mathcal{K}$ in $L^p(0,T;X_0)$ (in $C([0,T];X_0)$ 
if $p=\infty$). For example, we can assume that $\mathcal{K}$ 
is bounded in $L^1_{\mathrm{loc}}(0,T;X_1)$ and 
$$
\norm{u(\cdot+\Delta_t)-u}_{L^p(0,T-\Delta_t;X_0)}\to 0 
\quad \text{as $\Delta_t\to 0$, uniformly for $u\in \mathcal{K}$},
$$ 
cf.~\cite[Theorem 3]{Simon:1987vn}. We apply this result 
with $p=2$, $X_1=H^{1/2}(\Ge)$, $X_0=L^2(\Ge)$. Another result involves 
a third Banach space $X_{-1}$ (e.g.~$X_{-1}=H^{-1/2}(\Ge)$), such 
that $X_1\subset X_0\subset X_{-1}$ and $X_1$ is compactly embedded in $X_0$. 
Compactness of $\mathcal{K}$ in $L^p(0,T;X_0)$ follows 
if the set $\mathcal{K}$ is bounded in $L^p(0,T;X_1)$ and, as $\Delta_t\to 0$, 
$\norm{u(\cdot+\Delta_t)-u}_{L^p(0,T-\Delta_t;X_{-1})}\to 0$, 
uniformly for $u\in \mathcal{K}$ \cite[Theorem 5]{Simon:1987vn}.

\subsection{Two-scale convergence} \label{sec:twoScale}
Recall that $\Omega\subset \R^3$ denotes the entire 
(connected, bounded, open) cardiac domain, assumed  
to be of the form \eqref{def:domain}. 
The assumption \eqref{def:domain} simplifies the 
presentation. With mild modifications of the upcoming proofs, the results 
remain valid for general domains with Lipschitz boundary.
Let $Y$ be a reference unit cell in $\R^3$, which 
we fix to be the unit cube $Y := [0,1]^3$. 

Let $Y_i$ and $Y_e$ be the (disjoint, connected, open) 
intra and extracellular spaces within $Y$, separated 
by the cell membrane $\G$: 
$$
\overline{Y_i} \cup  \overline{Y_e}  = Y, 
\quad \G = \partial Y_i \setminus \partial Y.
$$
Denote by $K^{\e}$ the set of $k\in \Z^3$ for which 
$\bigcup_{k\in K^\e} \e\left(k+Y \right)=\overline{\Omega}$. 
We define the intracellular domain $\Omega_i^{\e}$, the extracellular 
domain $\Omega_e^{\e}$, and the cell membrane $\Ge$ as
\begin{equation}\label{domains-escaled}
	\Omega_j^{\e} = \bigcup_{k\in K^{\e}}\e(k+Y_j), 
	\quad j=i,e, \qquad 
	\Ge = \bigcup_{k\in K^\e} \e(k+\G).
\end{equation}
Both sets $\Oi,\Oe$ are connected Lipschitz domains, see Figure 1.1. 
Note however that it is impossible 
to have both $\Oi$ and $\Oe$ connected in a two-dimensional picture.

\begin{figure}
\begin{tikzpicture}[scale=0.17]

\newcommand{\s}{2}

\newcommand{\x}{20*\s}
\newcommand{\y}{0}
\newcommand{\Cx}{\x+5*\s}
\newcommand{\Cy}{\y+5*\s}
\renewcommand{\p}{3*\s}
\newcommand{\wg}{.5*\s}
\renewcommand{\lg}{2*\s}
\newcommand{\rr}{\s*10}

\pgfmathtruncatemacro{\rr}{10*\s};
\pgfmathtruncatemacro{\Dx}{\x+5*\s};
\pgfmathtruncatemacro{\Dy}{\y};
\pgfmathtruncatemacro{\Lx}{\x};
\pgfmathtruncatemacro{\Ly}{\y+5*\s};
\pgfmathtruncatemacro{\Ux}{\x+5*\s};
\pgfmathtruncatemacro{\Uy}{\y +10*\s};
\pgfmathtruncatemacro{\Rx}{\x+10*\s};
\pgfmathtruncatemacro{\Ry}{\y+5*\s};

\fill  [color=blue!30](\x,\y) rectangle (\x+\rr,\y+\rr);

\fill[color=red!30] (\Cx,\Cy) circle (\p);
\draw[thick,color=yellow] (\Cx,\Cy) circle (\p);

\fill [color=red!30] 
	(\Dx-\wg/2,\Dy) rectangle (\Dx+\wg/2 ,\Dy+\lg)
	(\Lx,\Ly-\wg/2) rectangle (\Lx+\lg ,\Ly+\wg/2)
	(\Ux-\wg/2,\Uy) rectangle (\Ux+\wg/2,\Uy-\lg)
	(\Rx,\Ry-\wg/2) rectangle (\Rx-\lg,\Ry+\wg/2);

\draw [thick,color=yellow] 
	(\Dx-\wg/2,\Dy) -- (\Dx-\wg/2 ,\Dy+\lg)
	(\Dx+\wg/2,\Dy) -- (\Dx+\wg/2 ,\Dy+\lg)
	(\Lx,\Ly-\wg/2) -- (\Lx+\lg ,\Ly-\wg/2)
	(\Lx,\Ly+\wg/2) -- (\Lx+\lg ,\Ly+\wg/2)
	(\Ux-\wg/2,\Uy)-- (\Ux-\wg/2,\Uy-\lg)
	(\Ux+\wg/2,\Uy) -- (\Ux+\wg/2,\Uy-\lg)
	(\Rx,\Ry-\wg/2)-- (\Rx-\lg,\Ry-\wg/2)
	(\Rx,\Ry+\wg/2)-- (\Rx-\lg,\Ry+\wg/2);
\draw [line width = 0.5mm] (\x,\y) rectangle (\x+\rr,\y+\rr);  


  \node [black] at (\Cx,\Cy) {\LARGE$Y_i$};
  \node [black] at (\Cx,\Cy+6*\s) {\LARGE$Y$};      
  \node [below right,black] at (55,5) {\LARGE$Y_e$};
  \node   at (\x+7*\s,7*\s) {\LARGE$\Gamma$};  
  \node   at (5*\s,11*\s) {\LARGE$\Omega= \Oi \cup \Oe\cup \Ge$};        
    
\node (A) at (11*\s,6*\s) {};
\node (B) at (19*\s,6*\s) {};
\draw[->, bend left=45, thick] 
       (A) edge[bend left ] node[above] {\Large$y=\frac{x}{\varepsilon}$} (B);

\draw [decorate,decoration={brace,amplitude=3*\s},xshift=5*\s,yshift=0pt,thick]
(10*\s,6*\s) -- (10*\s,4*\s)node [black,midway,xshift=14pt] {\footnotesize
\LARGE$\varepsilon$};


\renewcommand{\s}{2/5}
\renewcommand{\p}{3*\s}
\renewcommand{\wg}{.5*\s}
\renewcommand{\lg}{2*\s}
\renewcommand{\rr}{\s*10}
\foreach \i in {0,...,4}
{
	\foreach \j in {0,...,4}
	{
        \pgfmathtruncatemacro{\x}{\i *10*\s};
        \pgfmathtruncatemacro{\y}{\j*10*\s};

		\pgfmathtruncatemacro{\Cx}{\x+5*\s}
		\pgfmathtruncatemacro{\Cy}{\y+5*\s}
		
		\pgfmathtruncatemacro{\rr}{10*\s};
		\pgfmathtruncatemacro{\Dx}{\x+5*\s};
		\pgfmathtruncatemacro{\Dy}{\y};
		\pgfmathtruncatemacro{\Lx}{\x};
		\pgfmathtruncatemacro{\Ly}{\y+5*\s};
		\pgfmathtruncatemacro{\Ux}{\x+5*\s};
		\pgfmathtruncatemacro{\Uy}{\y +10*\s};
		\pgfmathtruncatemacro{\Rx}{\x+10*\s};
		\pgfmathtruncatemacro{\Ry}{\y+5*\s};
		\fill  [color=blue!30](\x,\y) rectangle (\x+\rr,\y+\rr);
		
		\fill[color=red!30] (\Cx,\Cy) circle (\p);
		\draw[thick,color=yellow] (\Cx,\Cy) circle (\p);

		\fill [color=red!30] 
			(\Dx-\wg/2,\Dy) rectangle (\Dx+\wg/2 ,\Dy+\lg)
			(\Lx,\Ly-\wg/2) rectangle (\Lx+\lg ,\Ly+\wg/2)
			(\Ux-\wg/2,\Uy) rectangle (\Ux+\wg/2,\Uy-\lg)
			(\Rx,\Ry-\wg/2) rectangle (\Rx-\lg,\Ry+\wg/2);
		
		\draw [thick,color=yellow] 
			(\Dx-\wg/2,\Dy) -- (\Dx-\wg/2 ,\Dy+\lg)
			(\Dx+\wg/2,\Dy) -- (\Dx+\wg/2 ,\Dy+\lg)
			(\Lx,\Ly-\wg/2) -- (\Lx+\lg ,\Ly-\wg/2)
			(\Lx,\Ly+\wg/2) -- (\Lx+\lg ,\Ly+\wg/2)
			(\Ux-\wg/2,\Uy)-- (\Ux-\wg/2,\Uy-\lg)
			(\Ux+\wg/2,\Uy) -- (\Ux+\wg/2,\Uy-\lg)
			(\Rx,\Ry-\wg/2)-- (\Rx-\lg,\Ry-\wg/2)
			(\Rx,\Ry+\wg/2)-- (\Rx-\lg,\Ry+\wg/2);
		\draw [line width = 0.1mm] (\x,\y) rectangle (\x+\rr,\y+\rr);  
	}
}
\draw [line width = 0.5mm] (0,0) rectangle (50*\s,50*\s);

\pgfmathtruncatemacro{\x}{4*10*\s};
\pgfmathtruncatemacro{\y}{2*10*\s};
\draw [line width = 0.5mm] (\x,\y) rectangle (\x+\rr,\y+\rr);


\end{tikzpicture}
\caption{The rescaled sets $\Oi$, $\Oe$, $\Ge$ (left) 
and the unit cell $Y$ (right).}
\end{figure}
 
To derive estimates for 
the microscopic model, we employ the following trace inequality 
for $\e$-periodic hypersurfaces:
\begin{equation}\label{trace}
	\e \norm{u|_{\Ge}}^2_{L^2(\Ge)} 
	\leq  C \left( \norm{u}^2_{L^2(\Oj)} 
	+ \e^2 \norm{\nabla u}^2_{L^2(\Oj)} \right),
	\quad u\in H^1(\Oj), \; j=i,e,
\end{equation}
for some constant $C$ independent of $\e$, 
cf.~\cite[Lemma 3]{Hornung:1991aa} 
or \cite[Lemma 4.2]{Marciniak-Czochra:2008aa}.

We need a uniform Poincar\'{e} inequality 
for perforated domains \cite{Cioranescu}.
\begin{lemma}
There exists a constant $C$, independent of $\e>0$, such that
\begin{equation}\label{Poincare}
	\norm{u -\frac{1}{\abs{\Oj}}
	\int_{\Oj} u \,dx }_{L^2\left(\Oj\right)} 
	\leq C \norm{\nabla u}_{L^2\left(\Oj\right)},  
\end{equation}
for all $u \in H^1\left(\Oj\right)$, $j=i,e$.
\end{lemma}

Estimate \eqref{Poincare} holds under mild regularity assumptions 
on the perforated domains; a Lipschitz boundary is more than 
sufficient (but connectedness is essential). 

Recall that a sequence $\seq{u^{\e}}_{\e>0} \subset L^2((0,T)\times \Omega)$ 
\textit{two-scale converges} to 
$u$ in $L^2((0,T)\times \O ;L^2_{\mathrm{per}}(Y))$ if 
\begin{equation}\label{def:two-scale}
	\int_0^T\int_{\O} u^{\e}(t,x)\p \left(t,x,\frac{x}{\e}\right) \, dx \, dt 
	\overset{\e\to 0}{\to} 
	\int_0^T \int_{\O} \int_Y u(t,x,y) \p(t,x,y) \, dy\, dx\, dt, 
\end{equation}
for all $\p\in C([0,T]\times \overline{\Omega};C_{\mathrm{per}}(Y))$. 
We express this symbolically as
$$
u^{\e}\overset{2}{\rightharpoonup} u. 
$$
By density properties, the convergence 
\eqref{def:two-scale} also holds for test functions 
$\p$ from $L^2((0,T)\times \Omega;C_{\mathrm{per}}(Y))$ 
\cite[p.~176]{Donato}.
 
The two-scale compactness theorem \cite{Allaire,Donato} 
is of fundamental importance.
\begin{theorem}[two-scale compactness]\label{twoL2}
Let $ \{u^{\e}\}_{\e>0}$ be a bounded sequence 
in $L^2((0,T)\times \O)$, $\norm{u^{\e}}_{L^2((0,T)\times \O)}
\leq C$ $\forall \e>0$. Then there exist a subsequence $\e_n\to 0$ and 
a function $u\in L^2( (0,T)\times \O;L^2(Y))$ such that $u^{\e_n}$ 
two-scale converges to $u$ as $n \rightarrow \infty$.
\end{theorem}

Consider a sequence $\{u_j^{\e}\}_{\e>0}$ of functions 
defined on the perforated domain $(0,T)\times \Oj$, $j=i,e$. 
We write $\widetilde{u_j^\e}$ for the 
zero-extension of $u_j^\e$ to $(0,T)\times\O$:
\begin{equation*}
	\widetilde{u_j^\e}(t,x) := 
	\begin{cases}
		u_j^\e(t,x) & \text{if} \, (t,x)\in (0,T)\times \Oj, \\
		0    & \text{if} \, (t,x) \in (0,T)\times\left(\O \setminus \Oj\right).
	\end{cases}
\end{equation*} 
By Theorem $\ref{twoL2}$, $\seq{\widetilde{u_j^\e}}_{\e>0}$
has a two-scale convergent subsequence, provided we know
that $\|u_j^{\e}\|_{L^2((0,T)\times \Oj)}\leq C$ $\forall \e>0$. 
However, this is not true in general for the gradient of $\widetilde{u_j^{\e}}$, 
even if $\norm{u_j^\e}_{L^2(0,T;H^1(\Oj))}\leq C$, since 
the extension by zero creates a discontinuity across $\Ge$. 
Instead the following statement holds true for the gradient: 

\begin{lemma}\label{twoH1perf}
Fix $j\in\{i,e\}$ and suppose $u^{\e}=$ $u_j^\e$  
satisfies $\norm{u^{\e}}_{L^2(0,T;H^1(\Oj))} \leq C$ 
$\forall \e>0$. Then there exist a subsequence $\e_n\to 0$ and  
functions $u\in L^2(0,T;H^1(\O))$, 
$u_1\in L^2((0,T)\times \O;H_{\mathrm{per}}^1(Y_j))$ 
such that as $n\to \infty$, 
\begin{equation*}
	\begin{split}
		\widetilde{u^{\e_n}} 
		& \overset{2}{\rightharpoonup}\mathbbm{1}_{Y_j}(y)u(t,x) 
		\; \text{in} \; L^2((0,T)\times \Omega;L_{\mathrm{per}}^2(Y)), \\
		\widetilde{\nabla u^{\e_n} } 
		& \overset{2}{\rightharpoonup}  
		\mathbbm{1}_{Y_j}(y) \big( \nabla_x u(t,x) 
		+ \nabla_y u_1(t,x,y) \big) \; \text{in} \; 
		L^2((0,T)\times\O;L_{\mathrm{per}}^2(Y)).
	\end{split}
\end{equation*}    
Here, $\mathbbm{1}_{Y_j}(y)$ denotes the characteristic function of $Y_j$, 
\begin{equation*}
	\mathbbm{1}_j(y) = 
	\begin{cases}
		1 \quad \text{if} \;\; y\in Y_j, \\
		0 \quad \text{if} \;\; y\notin Y_j.
	\end{cases}
\end{equation*}
\end{lemma}
For a proof of this lemma in the time independent 
case, see \cite[Theorem 2.9]{Allaire}. The extension to time 
dependent functions is straightforward.

There is an extension of two-scale convergence to 
periodic surfaces \cite{Allaire2}. 
Recall that a periodic surface $\Ge$ is given by
\begin{equation*}
 	\Ge:= \left\{ x\in \O \;\Big|\; \frac{x}{\e} \in k+\G \; 
	\text{for some } k\in \mathbb{Z}^3 \right\},
 \end{equation*}
where $\G \subset Y$ is a surface in the unit cell. 
Since $|\Ge| \sim \e^{-1}$, it is necessary to introduce a 
normalizing factor in the definition of two-scale convergence on surfaces. 

A sequence $\{v^{\e}\}_{\e>0}$ of functions in 
$L^2((0,T)\times \Ge)$ \textit{two-scale converges} 
to $v$ in $L^2((0,T)\times \O; L^2(\Gamma))$, written 
$$
v^{\e}\overset{2-\mathrm{S}}{\rightharpoonup} v,
$$
if, for all $\p \in C([0,T]\times \overline{\O};C_{\mathrm{per}}(\G))$,
\begin{align*}
	&\lim_{\e\to 0}\e \int_0^T\int_{\Ge} 
	v^{\e}(t,x) \p\left(t,x, \frac{x}{\e}  \right) \,dS(x) \, dt  
	\\ & \qquad \quad 
	= \int_0^T \int_{\O} \int_{\G} 
	v(t,x,y)\p(t,x,y) \, dS(y)\, dx\, dt.
\end{align*} 
As with \eqref{def:two-scale}, this convergence continues to hold 
for test functions $\p$ that belong to $L^2((0,T)\times\Omega;C_{\mathrm{per}}(\G))$.

There is a version of Theorem \ref{twoL2} for 
functions on periodic surfaces \cite{Allaire2}. 

\begin{theorem}[two-scale compactness on surfaces]\label{twoS}
Suppose $\{v^{\e}\}_{\e>0}$ is a sequence of functions in 
$L^2((0,T)\times \Gamma_\e)$ satisfying
\begin{equation}\label{two-scale-cond}
	\e \int_0^T\int_{\Ge} \abs{v^{\e}}^2\, dS\, dt \leq C,
\end{equation}
for some function $C$ that is independent of $\e>0$. 
Then there exist a subsequence $\e_n\to 0$ and a 
function $v \in L^2((0,T)\times \O ;L^2(\G))$ 
such that as $n\to \infty$, 
$$
v^{\e_n}\overset{2-\mathrm{S}}{\rightharpoonup} v.
$$
\end{theorem}

One can characterize the two-scale limit of traces of 
bounded sequences in $L^2(0,T;H^1(\O^{\e}))$ as 
the trace of the two-scale limit \cite{Allaire2}.   

\begin{lemma}\label{relation trace}
Fix $j\in\{i,e\}$. Suppose $u^{\e}=u_j^\e$  
satisfies $\norm{u^{\e}}_{L^2(0,T;H^1(\Oj))} 
\leq C$ $\forall \e>0$ and, cf.~Lemma \ref{twoH1perf}, 
$\widetilde{u^{\e}} \overset{2}{\rightharpoonup} 
\mathbbm{1}_{Y_j}(y)u(t,x)$ in $L^2((0,T)\times\O;L^2(Y))$. 
Let
$$
g^{\e}:= u^{\e}\big|_{\Ge} \in L^2((0,T)\times \Ge )
$$ 
be the trace of $u^{\e}$ on $\Ge = \partial \Oj \setminus \partial \O$. 
Then, up to a subsequence,
$$
g^{\e}  \overset{2-\mathrm{S}}{\rightharpoonup} 
g := \mathbbm{1}_{\Gamma}(y)u(t,x).
$$
\end{lemma}

\begin{remark}
In view of Lemma \ref{relation trace}, we have 
(in the sense of measures)
$$
\e \, u^{\e}\big|_{\Ge} \,dS\, dt \weakstar \abs{\G}\, u\, dx\, dt, 
\qquad
\widetilde{u^{\e}}\,dx\, dt \weakstar \abs{Y_j}\, u \,dx\, dt.
$$
\end{remark}

\subsection{Unfolding operators} \label{sec:unfolding}
An alternative approach to studying convergence on 
oscillating surfaces $\Ge$ is provided by the 
boundary unfolding operator \cite{Cioranescu}.
For any $x\in \R^3$, we have the decomposition 
$x = \lfloor x\rfloor + \{ x\}$, where 
$\lfloor x \rfloor\in \Z^3$ and $\{x\}\in [0,1]^3$ 
denotes the integer and fractional parts of $x$, respectively. 
For later use, note the following simple properties, which hold 
for any $x,\bar x\in \R^3$, $n\in \Z^3$:
$$
\lfloor x+n\rfloor=\lfloor x \rfloor+n, \quad
\{ x+n\} = \{ x\}, \quad
\lfloor x+\bar x\rfloor\le \lfloor x\rfloor
+\lfloor \bar x\rfloor+(1,1,1).
$$ 
Applying the above decomposition to $x/\e$ gives
$$
x = \e \left( \lfloor x/\e \rfloor +\{ x / \e\}\right), 
$$
where $\lfloor x/\e \rfloor\in \Z^3$, $\{ x / \e\}\in [0,1]^3$.

The \textit{boundary unfolding} operator $\Tau_{\e}^b$ is defined by 
\begin{equation}\label{unfolding}
	\begin{split}
		& \Tau_{\e}^b: L^2(0,T;L^2(\Ge))\to L^2(0,T;L^2(\O\times \G)),\\
		& \Tau_{\e}^b (v)(t,x,y) = 
		v\left(t,\e \left\lfloor \frac{x}{\e}\right\rfloor + \e y\right), 
		\qquad (t,x,y) \in (0,T)\times \O\times \G.
	\end{split}
\end{equation}
The advantage of the unfolding operator is that we can 
formulate questions of convergence in a fixed 
space $L^2(0,T;L^2(\O\times \G))$. All definitions and results 
in this section are formulated in $L^2$ spaces. Everything remains the 
same, however, if we replace $L^2$ by $L^p$ for any $p\in [1,\infty)$. 
We refer to \cite{Cioranescu} for the definition of the boundary 
unfolding operator and proofs of the properties listed next. 

The boundary unfolding operator $\Tau_{\e}^b$ is bounded, 
linear, and satisfies
\begin{equation}\label{Tb-product}
	\Tau_{\e}^b (v_1 v_2) = \Tau_{\e}^b (v_1) \Tau_{\e}^b(v_2), 
	\qquad  v_1,v_2 \in L^2(0,T;L^2(\Ge)).
\end{equation}
For any $Y$-periodic function $\psi\in L^2(\G)$, set 
$\psi_\e(x):=\psi(x/\e)$. Then 
\begin{equation*}
	\Tau_{\e}^b (\psi_\e)(x,y)=\psi(y), 
	\qquad x\in \O, \, y\in \G. 
\end{equation*}
For $v\in L^2(0,T;L^2(\Ge))$, we have the integration formula
\begin{equation}\label{IntegralTransform}
	\e \int_{\G^{\e}} v(t,x) \,dS(x)
	=   \int_{\O} \int_{\G} \Tau_{\e}^b(v)(t,x,y)
	\,dS(y) \, dx, 
\end{equation}
for a.e.~$t\in (0,T)$, thereby converting an 
integral over the oscillating set $\Ge $ to 
an integral over the fixed set $\O \times \G$. 
For $v\in L^2(0,T;L^2(\Ge))$,
\begin{equation}\label{Tb-L2Bound}
	\norm{\Tau_{\e}^b(v)}_{L^2(\O \times \G)} 
	= \e^{1/2} \norm{v}_{L^2(\Ge)},
\end{equation}
for a.e.~$t\in (0,T)$. For any $v\in L^2(0,T;L^2(\Ge))$, 
\begin{equation}\label{Tb-strongconv}
	\Tau_{\e}^b(v) \overset{\e\downarrow 0}{\to} 
	v \quad \text{in $L^2( \O \times \G)$},
\end{equation}
for a.e.~$t\in (0,T)$, and also in $L^2(0,T;L^2(\O \times \G))$.
Suppose $\{v^{\e}\}_{\e>0}$ is a sequence of functions 
in $L^2((0,T)\times \Ge)$ satisfying 
\eqref{two-scale-cond}. Then
\begin{equation}\label{weak-vs-twoscale}
	v^{\e}\overset{2-\mathrm{S}}{\rightharpoonup} v
	\; \Longleftrightarrow \;
	\Tau_{\e}^b(v^\e) \weak v \quad 
	\text{in $L^2((0,T)\times \O\times \G)$}. 
\end{equation}

We need also the \textit{unfolding} operators
linked to the domains $\Oi, \Oe$ \cite{Cioranescu}:
\begin{align*}
	&\Tau^j_{\e}:L^2\left((0,T)\times \Oj\right) 
	\to L^2(0,T;L^2(\O \times Y_j)), \quad j=i,e,
	\\ & \Tau^j_{\e} (u)(t,x,y) = 
	u\left(t,\e \left\lfloor \frac{x}{\e}\right\rfloor + \e y\right), 
	\quad (t,x,y) \in (0,T)\times \O\times Y_j.
\end{align*}
The unfolding operator $\Tau^j_{\e}$ maps functions 
defined on the oscillating set $(0,T)\times \Oj$ into 
functions defined on the fixed domain $(0,T)\times \O \times Y_j$.  
The operator $\Tau_{\e}^j$ is bounded, linear and satisfies
\begin{equation*}
	\Tau_{\e}^j (uv) = \Tau_{\e}^j (u) \Tau_{\e}^j(v), 
	\qquad  u,v \in L^2\left( (0,T)\times \Oj \right).
\end{equation*}
For any $Y$-periodic function $\psi\in L^2(Y_j)$, set 
$\psi_\e(x):=\psi(x/\e)$. Then 
\begin{equation*}
	\Tau_{\e}^j (\psi_\e)(x,y)=\psi(y), 
	\qquad x\in \O, \, y\in Y_j. 
\end{equation*}
For $u \in L^2(0,T;L^2(\Oj))$, we have the integration formula
\begin{equation*}
	 \int_{\Oj} u(t,x) \,dx\, dt 
	 =\int_{\O} \int_{Y_j} \Tau_{\e}^j(u)(t,x,y)\,dy \, dx \, dt,
\end{equation*}
for a.e.~$t\in (0,T)$. The integration formula implies
\begin{equation}\label{LpBound2}
	\norm{\Tau_{\e}^j(u)}_{L^2(\O \times Y_j)} 
	=\norm{u}_{L^2(\Oj)},
\end{equation}
for a.e.~$t\in (0,T)$. Let $u \in L^2(0,T;H^1(\Oj))$. Then 
\begin{equation*}
	\nabla_y \Tau_{\e}^j (u) 
	=\e \Tau_{\e}^j(\nabla u),
	\quad \text{a.e.~in $(0,T)\times \O\times Y_j$},
\end{equation*}
and hence $\Tau_{\e}^j (u)\in L^2(\O;H^1(Y_j))$, for a.e.~$t\in (0,T)$:
\begin{equation}\label{unfoldingH1}
	\norm{\nabla_y \Tau_{\e}^j(u)}_{L^2(\O \times Y_j))} 
	=\e \norm{\nabla u}_{L^2(\Oj)}.
\end{equation}

The unfolding operators $\Tau_{\e}^b$ and $\Tau_{\e}^j$ 
are related in the following sense:
\begin{equation}\label{unfoldingTrace}
	\Tau_{\e}^b ( u|_{\Ge}) 
	= \Tau_{\e}^j(u)|_{\G}, 
	\qquad u \in L^2(0,T;H^1(\Oj)),\quad j=i,e,
\end{equation}
for a.e.~$t\in (0,T)$. Combining \eqref{unfoldingTrace}, 
\eqref{LpBound2}, \eqref{unfoldingH1}, and the trace 
inequality \eqref{trace1} in $H^1(\Oj)$, we obtain
\begin{equation}\label{unfoldingH1/2}
	\begin{split}
		& \norm{\Tau_{\e}^b(u|_{\Ge})}_{L^2(\O;H^{1/2}(\G))}^2 
		\\ & \qquad
		\le  C \left(\norm{\Tau_{\e}^j(u)}_{L^2(\O;L^2(Y_j))}^2
		+ \norm{\nabla_y \Tau_{\e}^j(u)}_{L^2(\O;L^2(Y_j))}^2\right) 
		\\ & \qquad =  C\left(\norm{u}_{L^2(\Oj)}^2
		+\e^2 \norm{\nabla u}_{L^2(\Oj)}^2\right),
	\quad j=i,e,
\end{split}
\end{equation}
for a.e.~$t\in (0,T)$, where the constant 
$C$ is independent of $\e$ and $t$.
Whenever it is convenient, we will
write $\Tau_{\e}^b (u)$ instead 
of $\Tau_{\e}^b(u|_{\Ge})$.

Next, we consider the local average (mean in the cells) operator 
\begin{equation*}
	\mathcal{M}^j_{\e}(u)(t,x) 
	= \int_{Y_j} \Tau_{\e}^{j}(u)(t,x,y)\,dy, 
	\quad (t,x)\in (0,T)\times \Omega,
	\qquad j=i,e,
\end{equation*}
and the piecewise linear interpolation operator \cite[Definition 2.5]{Cioranescu}
\begin{equation}\label{Qdefinition}
	\begin{split}
		& Q_{\e}^{j}:L^2(0,T;H^1(\Oj)) \to L^2(0,T;H^1(\O)), \quad j=i,e,
		\\ & 
		Q_{\e}^{j}(u) \;\; 
		\text{is the $Q_1$-interpolation (in $x$) of $\mathcal{M}^j_{\e}(u)$}. 
	\end{split}
\end{equation}
Given the Lipschitz regularity of $Y_j$, the interpolation operator 
$Q_\e^j$ satisfies the following estimates 
\cite[Propositions 2.7 and 2.8]{Cioranescu}:
\begin{equation}\label{Qestimates}
	\begin{split}
		& \norm{Q_{\e}^j(u)-u}_{L^2((0,T)\times \Oj)} 
		\leq C \e \norm{\nabla u}_{L^2((0,T)\times \Oj)}, 
		\\ & 
		\norm{\nabla Q_{\e}^j(u)}_{L^2((0,T)\times \Oj)} 
		\leq C \norm{\nabla u}_{L^2((0,T)\times \Oj)},
	\end{split}
\end{equation} 
where $C$ is a constant that is independent of $\e$.

\section{Microscopic bidomain model}\label{sec:wellposedness}

In this section we present a relevant notion of (weak) solution 
for the microscopic problem \eqref{micro}, \eqref{eq:ib-cond}, along with 
an accompanying existence theorem. We also 
derive some ``$\e$-independent" a priori estimates, 
which are used later to extract two-scale convergent subsequences.  
 
\subsection{Assumptions on the data} \label{sec:assumptions}
We impose the following set of assumptions on 
the ``membrane" functions $I,H$:

$\bullet$ Generalized FitzHugh-Nagumo model: For $v,w\in \R$,
\begin{equation}\label{GFHN}
	\begin{split}
		& I(v, w) = I_1 (v) + I_2(v)w,
		\quad  H(v, w) = h(v) + c_{H,1}w,\\
		& \text{where } I_1 , I_2, h \in C^1(\R),\, c_{H,1}\in \R, \text{ and} \\ 
		&\abs{I_1(v)} \leq c_{I,1}\left(1+\abs{v}^3\right), 
		\quad I_1(v)v \geq c_I \abs{v}^4-c_{I,2}\abs{v}^2, 
		\\ & 
		I_2(v) = c_{I,3}+ c_{I,4}v, 
		\quad \abs{h(v)}\leq c_{H,2}\left(1+\abs{v}^2\right),
	\end{split}
	\tag{\textbf{GFHN}}
\end{equation}
for some constants $c_I > 0$ and $c_{I,1}, c_{I,2}, 
c_{I,3}, c_{I,4}, c_{H,2}\ge 0$. 

The classical FitzHugh-Nagumo model corresponds to
\begin{equation}\label{FHN}
	I(v,w)=v(v-a)(v-1)+w, \qquad H(v,w)=\epsilon (k v -w),
\end{equation}
where $a\in (0,1)$ and $k,\epsilon >0$ are constants. 

Repeated applications of Cauchy's inequality yields
\begin{equation}\label{IH-conseq}
	vI(v,w)-wH(v,w) 
	\ge \gamma \abs{v}^4-\beta \left(\abs{v}^2+\abs{w}^2\right),
\end{equation}
for some constants $\gamma>0$ and $\beta\ge 0$. 
This inequality will be used to bound the 
transmembrane potential in the $L^4$ norm.

Consider a quadratic matrix $A$, which always
can be written as the sum of its symmetric 
part $\frac12 (A+A^\top)$ and 
its skew-symmetric part $\frac12 (A-A^\top)$. 
Recall that in a quadratic form $z\mapsto Az\cdot z$ the 
skew-symmetric part does not contribute. 
Therefore, letting $\lambda_{\text{min}}$ 
and $\lambda_{\text{max}}$ denote respectively 
the minimum and maximum eigenvalues of the symmetric 
part of $A$, we have 
$$
\lambda_{\text{min}} \abs{z}^2\le  Az\cdot z 
\le \lambda_{\text{max}} \abs{z}^2, \qquad \forall z.
$$
For $\mu>0$, consider the function 
$F^\mu:\R^2\to \R^2$ defined by
$$F^\mu(z)
=
\begin{pmatrix}
\mu I(z)\\
-H(z)	
\end{pmatrix}.$$
Denote by $\lambda^\mu_{\text{min}}(z)$, 
$\lambda^\mu_{\text{max}}(z)$ 
the minimum, maximum eigenvalues of the 
symmetric part of the matrix $\nabla F^\mu(z)$. 
To ensure that weak solutions are unique, we need an additional 
assumption on $I, H$ expressed via $F^\mu$ \cite[p.~479]{Bourgault}: 
$\exists \mu,\lambda>0$ such that  
\begin{equation}\label{uniq-cond1}
	\lambda^\mu_{\text{min}}(z)\ge \lambda, \quad 
	\forall z\in \R^2.
\end{equation}
One can verify that the FitzHugh-Nagumo model \eqref{FHN} 
obeys \eqref{uniq-cond1} (with $\mu=\e k$).

A consequence of \eqref{uniq-cond1} is that 
$\nabla F^\mu(\bar z) z \cdot z \ge \lambda \abs{z}^2$ for 
all $\bar z, z\in \R^2$. Therefore, writing
$$
F^\mu(z_2)-F^\mu(z_1) 
=\int_0^1 \nabla F^\mu(\theta z_2+(1-\theta)z_1)(z_2-z_1)\,d\theta,
$$
it follows that
$$
\left(F^\mu(z_2)-F^\mu(z_1)\right)
\cdot (z_2-z_1)\geq
-\lambda\abs{z_2-z_1}^2.
$$
More explicitly, assumption \eqref{uniq-cond1} implies 
the following ``dissipative structure" on a suitable 
linear combination of $I$ and $H$:
\begin{align*}
	&\mu \left(I(v_2,w_2)-I(v_1,w_2)\right)(v_2-v_1) 
	-\left(H(v_2,w_2)-H(v_1,w_1)\right)(w_2-w_1) 
	\\ & \qquad 
	\geq -\lambda\left(\abs{v_2-v_1}^2+\abs{w_2-w_1}^2\right), 
	\qquad \forall v_1,v_2,w_1,w_2 \in \R.
\end{align*}
This inequality implies the $L^2$ stability (and thus 
uniqueness) of weak solutions.

\begin{remark}
There are many membrane models of cardiac cells \cite{ColliBook,Sundnes}. 
We utilize the FitzHugh-Nagumo model \cite{FitzHugh1955}, 
which is a simplification of the Hodgin-Huxley model 
of voltage-gated ion channels. It is possible to treat other membrane models 
by blending the arguments used herein with those found in 
\cite{Boulakia2008,Bourgault,ColliBook,Sundnes,Veneroni,Veneroni:2009aa}.
\end{remark}

As a natural assumption for homogenization, we assume 
that the $\e$-dependence of the conductivities 
$\sigma_j^\e$ ($j=i,e$), the applied currents 
$s^{\e}_j$ ($j=i,e$), and the initial 
data $v_0^{\e}, w_0^{\e}$ decouples into 
a ``fast" and a ``slow" variable:
\begin{equation}\label{fastSlow}
	\begin{split}
		&\sigma_j^{\e}(x) = \sigma_j\left(x,\frac{x}{\e}\right),
		\quad 
		s^{\e}_j(x) = s_j\left(x,\frac{x}{\e}\right), 
		\\ & 
		v_0^{\e}(x) = v_0\left(x,\frac{x}{\e}\right),
		\quad
		w_0^{\e}(x) = w_0\left(x,\frac{x}{\e}\right),
	\end{split}
\end{equation} 
for some fixed functions $\sigma_j(x,y),s_j(x,y),v_0(x,y),w_0(x,y)$ 
that are $Y$-periodic in the second argument. 
 
The conductivity tensors are assumed to be bounded and continuous,
\begin{equation}\label{eq:matrix1}
	\sigma_j(x,y) \in L^\infty(\O\times Y_j)\cap 
	C\left(\Omega; C_{\mathrm{per}}(Y_j)\right), 
\end{equation}
and satisfy the usual ellipticity 
condition, i.e., there exists $\alpha >0$ such that
\begin{equation}\label{eq:matrix2}
	\eta \cdot \sigma_j(x,y) \eta \geq \alpha |\eta|^2,  \quad 
	\forall \eta \in \R^3, \; \forall (x,y)\in \O\times Y_j,
\end{equation}
for $j=i,e$. Finally, we assume that each $\sigma_j$ is symmetric: 
$\sigma_j^T= \sigma_j$.

The regularity assumption \eqref{eq:matrix1} 
implies that $\sigma_j$ is an admissible test function 
for two-scale convergence \cite{Allaire}, which means that
$$
\lim_{\e \rightarrow 0}\int_{\Oj}  \sigma^{\e}_j(x) \p^{\e}(x) \, dx 
= \int_{\O \times Y_j} \sigma(x,y)\p(x,y) \, dx\, dy, 
\qquad j=i,e,
$$
for every two-scale convergent sequence $\{\p^{\e}\}_{\e>0}$, 
$\p^{\e}\overset{2}{\rightharpoonup} \p$. 
This convergence still holds if the second part of \eqref{eq:matrix1} 
is replaced by $\sigma_j \in L^2\left(\Omega; C_{\mathrm{per}}(Y_j)\right)$.

For the stimulation currents we assume the compatibility condition
\begin{equation}\label{assumptionCompatibility}
	\sum_{j=i,e} \int_{\Oj} s_j^\e\,dx =0,
\end{equation} 
and the boundedness in $L^2$:
\begin{equation}\label{assumptionStimulation}
	\norm{s_j^{\e}}_{L^2((0,T)\times \Oj)}\leq C, 
	\qquad j=i,e,
\end{equation}
which, in view of \eqref{fastSlow}, is 
guaranteed if we take 
\cite[p.~174]{Donato}
\begin{equation}\label{sj-ass}
	s_j \in L^2(\O ; C_{\text{per}}(Y_j)).
\end{equation}
Similarly, we assume that 
\begin{equation}\label{assumptionInitial}
	v_0, w_0 \in L^2(\O ; C_{\text{per}}(\G)).
\end{equation}

Throughout this paper we denote by $C$ a generic constant, not depending 
on the parameter $\e$. The actual value 
of $C$ may change from one line to the next.

\subsection{Weak solutions}
Testing \eqref{micro} against appropriate functions 
$\p_i,\p_e,\p_w$ we obtain the weak formulation 
of the microscopic bidomain model \eqref{micro}, \eqref{eq:ib-cond}, 
cf.~\cite{Colli,Veneroni} for details. 
We note that the terms involving the boundary $\partial \O$ 
vanish due to the Neumann boundary condition \eqref{eq:ib-cond}.

\begin{definition}[weak formulation of microscopic system]\label{weak}
A weak solution to \eqref{micro}, \eqref{eq:ib-cond} is a 
collection $(u_i^{\e},u_e^{\e},v^{\e},w^{\e})$ of functions 
satisfying the following conditions:
\begin{enumerate}[i]
	\item (algebraic relation).
	\begin{equation}\label{ve-def}
		v^{\e}= u_i^{\e}\big|_{\Ge} - u_e^{\e}\big|_{\Ge}
		\quad \text{a.e.~on $(0,T)\times\Ge$}.
	\end{equation}

\item (regularity). 
	\begin{align*}
		& u_j^{\e} \in L^2(0,T;H^1(\Oj)), \quad j=i,e, \\
		& \int_{\Oe} u_e^\e(t,x)\, dx = 0, \quad t\in (0,T), \\
		& v^{\e} \in L^2(0,T; H^{1/2}(\Ge))\cap L^4((0,T)\times \Ge), \\ 
		& \dt v \in L^2(0,T;H^{-1/2}(\Ge))+L^{4/3}((0,T)\times \Ge), \\
		& w^{\e} \in H^1(0,T;L^2(\Ge)).
	\end{align*}
		
\item (initial conditions).
	\begin{equation}\label{eq:init}
		v^{\e}(0)=v_0^{\e}, \quad w^{\e}(0) = w_0^{\e}.
	\end{equation}

\item (differential equations).
	\begin{equation}\label{weak_i}
		\begin{split}
			& \e \int_0^T \left\langle \dt v^{\e},
			\p_i\right\rangle \, dt 
			+  \int_0^T \int_{\Oi}\sigma_i^{\e} \nabla u_i^{\e} 
			\cdot \nabla \p_i\, dx \, dt
			\\ & \qquad\quad  
			+ \e \int_0^T \int_{\Ge} I(v^{\e},w^{\e})\p_i \,dS \, dt 
			= \int_0^T \int_{\Oi} s^{\e}_{i} \p_i \, dx \, dt,
		\end{split}
	\end{equation}
	\begin{equation}\label{weak_e}
		\begin{split}
			& \e \int_0^T \left\langle\dt v^{\e},  
			\p_e\right\rangle \, dt
			- \int_0^T\int_{\Oe}\sigma_e^{\e} \nabla u_e^{\e} 
			\cdot \nabla \p_e \, dx \, dt
			\\ & \qquad\quad  
			+ \e \int_0^T \int_{\Ge} I(v^{\e},w^{\e})\p_e \,dS \, dt  
			= -\int_0^T \int_{\Oe} s^{\e}_{e} \p_e \, dx \, dt,
		\end{split}
	\end{equation}
	\begin{equation}\label{weak_w}
		\int_{\Ge} \dt w^{\e} \p_w\, dx 
		= \int_{\Ge} H(v^{\e},w^{\e}) \p_w \,dS,
	\end{equation}
\end{enumerate}
for all $\p_j\in L^2(0,T;H^1(\Oj))$ with 
$\p_j\in L^4((0,T)\times \Ge)$ ($j=i,e$), and for 
all $\p_w\in L^2(0,T;L^2(\Ge))$. 
\end{definition}

\begin{remark}\label{rem:well-defined}
In \eqref{weak_i}, \eqref{weak_e} we use $\left\langle \cdot, \cdot \right \rangle$ 
to denote the duality pairing between $H^{-1/2}(\Ge)+L^{4/3}(\Ge)$ 
and $H^{1/2}(\Ge)\cap L^4(\Ge)$. For a motivation of the regularity conditions 
in Definition \ref{weak}, see Remark \ref{rem:motivate-regularity} below. 

By the embedding \eqref{cont-embeding}, $v^{\e},w^{\e}\in C\left(0,T;L^2(\Ge) \right)$ 
and therefore the pointwise evaluations $v(0)$, $w(0)$ in \eqref{eq:init} are well defined.  
The time derivative $\partial_t v^\e$ is a distribution belonging to 
$L^2(0,T;H^{-1/2}(\Ge))+L^{4/3}((0,T)\times \Ge)$ 
with initial values $v^\e_{0}$, so that the integration-by-parts 
formula \eqref{eq:int-by-parts-general} holds. Consequently, we may replace 
$$
\e \int_0^T \left\langle \dt v^{\e},
\p_j\right\rangle \, dt, \quad j=i,e,
$$
in \eqref{weak_i} and \eqref{weak_e} by 
\begin{equation}\label{weak-form-ibp}
	-\e \int_{0}^{T} \int_{\Ge} v^\e \, \partial_{t} 
	\p_j \, dS(x)\,dt - \e\int_{\Ge}v^\e_0\, \p_j(0) \, dS(x)\,dt,
\end{equation}
for all test functions $\p_j \in C^\infty_0([0,T)\times \Oj)$, or  for all 
$\p_j\in L^2(0,T;H^1(\Oj))$ such that $\dt \p_j\in L^2((0,T)\times \Ge)$, 
$\p_j\in L^4((0,T)\times \Ge)$, and $\p_j(T)=0$. Later, when passing to 
the limit $\e\to 0$ in \eqref{weak_i} and \eqref{weak_e}, 
we make use of the form \eqref{weak-form-ibp}.
\end{remark}

\begin{remark}\label{weak:welldef}
Consider a weak solution $(u_i^\e,u_e^\e,v^{\e},w^{\e})$ 
according to Definition \ref{weak}. 
Thanks to \eqref{trace1}, the trace of $\p_j \in L^2(0,T;H^1(\Oj))$
belongs to $L^2(0,T;H^{1/2}(\Ge))$. By the 
Sobolev inequality \eqref{SobolevInequality}, the trace of 
$\p_j$ belongs also to $L^2(0,T;L^4(\Ge))$. 
In Definition \ref{weak} we ask additionally that 
the trace of $\p_j$ belongs to $L^4((0,T)\times \Ge)$ to ensure that the 
surface terms in \eqref{weak_i}, \eqref{weak_e} are well-defined

Indeed, $\int_0^T \left\langle \dt v^{\e}, \p_j\right\rangle \, dt$ 
is well-defined for such $\p_j$. 
Moreover,
\begin{align*}
	J:=\abs{\int_0^T\int_{\Ge} I(v^{\e},w^{\e}) \p_j\, dS\, dt}
	& \leq \norm{I(v^{\e},w^{\e})}_{L^{4/3}((0,T)\times \Ge)} 
	\norm{\p_j}_{L^4((0,T)\times \Ge)}.
\end{align*}
For the membrane model \eqref{GFHN}, the 
growth condition on $I$ implies
\begin{equation}\label{I43}
	\abs{I(v,w)}^{\frac43} 
	\le C_I\left(1+\abs{v}^4+ \abs{w}^2\right),
\end{equation}
and therefore $\norm{I(v^{\e},w^{\e})}_{L^{4/3}((0,T)\times \Ge)}^{4/3}<\infty$. 
Consequently, $J<\infty$. 

We actually have a more precise bound. 
As $H^{1/2}(\Ge)\subset L^4(\Ge)$, we have 
$H^{-1/2}(\Ge)\subset L^{4/3}(\Ge)$ and 
$$
\norm{I(v^{\e},w^{\e})}_{H^{-1/2}(\Ge)}^{4/3}
\le \tilde C \norm{I(v^{\e},w^{\e})}_{L^{4/3}(\Ge)}^{4/3}
\le C \left(1+\norm{v^{\e}}_{L^4(\Ge)}^4
+\norm{w^{\e}}_{L^2(\Ge)}^2\right).
$$
Integrating this over $t\in (0,T)$ yields 
$I(v^{\e},w^{\e})\in L^{4/3}(0,T;H^{-1/2}(\Ge))$.

The integral on the right-hand side of \eqref{weak_w} 
can be treated similarly, since $\abs{H(v,w)}^2 
\le C\left(\abs{v}^4+\abs{w}^2\right)$. 
The remaining integrals are trivially well defined.
\end{remark}

\subsection{Existence of solution and a priori estimates}
Existence and uniqueness results for certain classes of 
membrane models have been established in \cite{Colli,Veneroni}.
These works employ the variable $v=u_i-u_e$ to convert \eqref{micro} into 
a non-degenerate ``abstract" parabolic equation. 
The authors in \cite{Colli} then appeal to the theory of  
variational inequalities, whereas in \cite{Veneroni} the 
Schauder fixed point theorem is applied 
to conclude the existence of a solution.

The following theorem can be proved 
by adapting arguments found in \cite{Colli,Veneroni} (see 
also \cite{Grandelius:2017aa}), or those 
utilized in \cite{Bourgault,Karlsen,Andreianov:2010uq} for 
the (macroscopic) bidomain model.

\begin{theorem}[existence of weak solution for microscopic system]\label{thm:micro}
Fix any $\e>0$, and suppose \eqref{GFHN}, \eqref{def:domain},
\eqref{fastSlow}, \eqref{eq:matrix1}, \eqref{eq:matrix2}, 
\eqref{assumptionCompatibility}, \eqref{sj-ass}, and 
\eqref{assumptionInitial} hold. Then the microscopic bidomain 
model \eqref{micro}, \eqref{eq:ib-cond} possesses a weak 
solution $(u_i^{\e},u_e^{\e},v^{\e},w^{\e})$ (in the sense 
of Definition \ref{weak}). This weak solution 
satisfies the a priori estimates collected in 
Lemma \ref{lemma:energy} below.
\end{theorem}

\begin{remark}\label{rem:uniq}
The unknowns $u_i^\e,u_e^\e$ are determined 
uniquely up to a constant. We can fix this constant by imposing the 
normalization condition $\int u_e^\e\, dx = 0$, as 
is done in Definition \ref{weak}. Weak solutions are 
unique (and $L^2$ stable) provided we impose the structural 
condition \eqref{uniq-cond1} on the membrane functions $I,H$.
\end{remark}

The next lemma, which is utilized below to derive 
some a priori estimates, is a consequence of the 
uniform Poincar\'{e} inequality \eqref{Poincare} 
and the trace inequality for $\e$-periodic surfaces \eqref{trace}. 
A similar result is used in \cite{Pennacchio2005}.

\begin{lemma}\label{L2-uie-est}
Let $u_j\in H^1(\Oj)$, $j=i,e$, $\int_{\Oe} u_e \,dx =0$, 
and set $v:=u_i\big|_{\Ge}-u_e\big|_{\Ge}$.
There is a positive constants $C$, independent of $\e$, such that
\begin{align*} 
	\norm{u_i}_{L^2(\Oi)}^2 
	\leq C \left(\e\norm{v}_{L^2(\Ge)}^2
	+\sum_{j=i,e}\int_{\Oj} \abs{\nabla u_j}^2\,dx  \right).
\end{align*}
\end{lemma}

\begin{proof}
First, since $\int u_e =0$, 
\begin{equation}\label{poincare-ue}
	\norm{u_e}_{L^2(\Oe)}^2
	\overset{\eqref{Poincare}}{\le} 
	C_1\int_{\Oe} \abs{\nabla u_e}^2\, dx.
\end{equation}
To estimate the $L^2$ norm of $u_i$, write 
$u_i=\bar{u}_i+\tilde{u}_i$, where 
$\bar{u}_i:=\frac{1}{\abs{\Oi}}\int_{\Oi} u_i\, dx$ 
is constant in $\Oi$ and $\tilde{u}_i:=u_i-\bar{u}_i$ 
has zero mean in $\Oi$. Clearly,
$$
\norm{u_i}_{L^2(\Oj)}^2 
=\norm{\bar{u}_i}_{L^2(\Oi)}^2
+\norm{\tilde{u}_i}_{L^2(\Oi)}^2.
$$
In view of the Poincar\'{e} inequality \eqref{Poincare},
\begin{equation}\label{poincare-tui}
	\norm{\tilde{u}_i}_{L^2(\Oi)}^2
	\le \tilde C_1 \int_{\Oi} \abs{\nabla \tilde u_i}^2\, dx
	=\tilde C_1 \int_{\Oi} \abs{\nabla u_i}^2\, dx.
\end{equation}
Let us bound 
$\norm{\bar{u}_i}_{L^2(\Oi)}^2
=\frac{\abs{\Oi}}{\abs{\Ge}} \norm{\bar{u}_i}_{L^2(\Ge)}^2$. 
Since $\abs{\Ge}=\abs{K^{\e}} \int_{\e\G} \,dS 
= \e^{-3}\e^2 \int_{\G}\,dS$ (recall \eqref{domains-escaled}), 
$\e \left|\Ge\right| \geq c >0$. 
Because of this and $\abs{\Oi}\le \abs{\O}$, 
$$
\norm{\bar{u}_i}_{L^2(\Oi)}^2
\le \bar C_1 \e \norm{\bar{u}_i}_{L^2(\Ge)}^2.
$$
Noting that
$$
\abs{\bar{u}_i}^2
\le \bar C_2\left(\abs{u_i - u_e}^2 +  
\abs{\tilde u_i}^2 + \abs{u_e}^2\right),
$$
we obtain 
\begin{align*}
	\norm{\bar{u}_i}_{L^2(\Oi)}^2
	& \le \bar C_3 
	\left( \e  \norm{v}_{L^2(\Ge)}^2
	+ \e \norm{\tilde u_i}_{L^2(\Ge)}^2
	+ \e \norm{u_e}_{L^2(\Ge)}^2 \right)
	\\ & 
	\overset{\eqref{trace}}{\le} 
	\bar C_3 \e\norm{v}_{L^2(\Ge)}^2
	\\ & \qquad \quad 
	+ \bar C_4 \left(
	\norm{\tilde u_i}_{L^2(\Oi)}^2+
	\e^2 \norm{\nabla \tilde u_i}_{L^2(\Oi)}^2\right)
	\\ & \qquad \quad \quad
	+ \bar C_4 \left(
	\norm{u_e}_{L^2(\Oe)}^2+
	\e^2 \norm{\nabla u_e}_{L^2(\Oe)}^2\right)
	\\ & \le 
	\bar C_3 \e\norm{v}_{L^2(\Ge)}^2
	+ \bar C_5\norm{\nabla u_i}_{L^2(\Oi)}^2
	+ \bar C_5 \norm{\nabla u_e}_{L^2(\Oe)}^2,
\end{align*}
where the final inequality is a result 
of \eqref{poincare-ue} and \eqref{poincare-tui}.
\end{proof}

For the sake of the upcoming homogenization result, 
we now list some precise ($\e$-independent) 
a priori estimates.

\begin{lemma}[basic estimates for microscopic system]\label{lemma:energy}
Referring Theorem \ref{thm:micro}, the weak solution 
$(u_i^{\e},u_e^{\e},v^{\e},w^{\e})$ of 
\eqref{micro}, \eqref{eq:ib-cond} satisfies 
the a priori estimates
\begin{equation}\label{eq:main-est}
	\begin{split}
		& (a)\; \norm{\nabla u_j^\e}_{L^2\left((0,T)\times \Oj\right)} \le C,
		\quad j=i,e,
		\\
		& (b)\; \norm{u_j^\e}_{L^2\left((0,T)\times \Oj\right)} \le C, 
		\quad j=i,e,
		\\ & 
		(c)\; \e^{1/2} \norm{v^\e}_{L^{\infty}(0,T; L^2(\Ge))} \le C,
		\\ & 
		(d)\; \e^{1/4}\norm{v^\e}_{L^{4}((0,T)\times \Ge)}\le C,
		\\ & 
		(e)\; \e^{1/2}\norm{w^\e}_{L^{\infty}(0,T;L^2(\Ge))}\le C, 
		\\ &
		(f)\; \e^{1/2}\dt w^\e \in L^2(0,T;L^2(\Ge)),
	\end{split}
\end{equation}
where $C$ is a positive constant independent of $\e$.
\end{lemma}

\begin{proof}
We only \textit{outline} a proof of these (mostly standard) estimates. 

Specifying $(\p_i, \p_e,\p_w) = (u_i^{\e},-u_e^{\e},w^{\e})$ 
as test functions in \eqref{weak_i}, \eqref{weak_e}, and \eqref{weak_w}, 
adding the resulting equations, applying the 
chain rule \eqref{eq:chain-rule}, and using \eqref{eq:matrix2}, we obtain
\begin{equation}\label{vL2}
	\begin{split}
		&\frac{\e}{2}\frac{d}{dt} \left( \norm{v^{\e}}_{L^2(\Ge)}^2 
		+\norm{w^{\e}}_{L^2(\Ge)}^2 \right) 
		+\alpha \sum_{j=i,e}  \int_{\Oj} \abs{\nabla u_j^{\e}}^2\, dx, 
		\\ &\qquad \quad 
		+ \e \int_{\Ge} \left(I(v^{\e},w^{\e}) v^{\e}
		-H(v^{\e},w^{\e})w^{\e} \right)\,dS 
		\le \sum_{j=i,e}\int_{\Oj} s_{j}^\e u_j^{\e} \, dx.
	\end{split}
\end{equation}

Thanks to \eqref{IH-conseq},
\begin{equation}\label{ionL2}
	\begin{split}
		& \int_{\Ge} \left(I(v^{\e},w^{\e}) v^{\e} 
		- H(v^{\e},w^{\e})w^{\e} \right)\,dS 
		\\ & \qquad\quad
		\geq \gamma \norm{v^{\e}}_{L^4(\Ge)}^4
		- \beta\left( \norm{v^{\e}}_{L^2(\Ge)}^2
		+\norm{w^{\e}}_{L^2(\Ge)}^2 \right),
	\end{split}
\end{equation}
for some $\e$-independent constants $\gamma>0$ and $\beta\ge 0$.

Using Cauchy's inequality (``with $\delta$") , the 
source term can be bounded as 
\begin{equation}\label{source}
	\begin{split}
		\abs{\sum_{j=i,e}\int_{\Oj} s_{j}^{\e} u_j^{\e} \, dx} 
		& \leq \sum_{j=i,e} 
		\left(\norm{u_j^{\e}}_{L^2(\Oj)} \norm{s_j^{\e}}_{L^2(\Oj)}\right) 
		\\ & 
		\leq \delta \sum_{j=i,e} \norm{u_j^{\e}}_{L^2(\Oj)}^2 
		+ C_{\delta} \sum_{j=i,e}\norm{s_j^{\e}}^2_{L^2(\Oj)},
	\end{split}
\end{equation}
with $\delta>0$ small and $C_\delta$ independent of $\e$. 
Lemma \ref{L2-uie-est} and \eqref{poincare-ue} ensure that
\begin{equation}\label{source-part}
	\sum_{j=i,e} \norm{u_j^{\e}}_{L^2(\Oj)}^2
	\leq C_1\left(\e  \norm{v^{\e}}_{L^2(\Ge)}^2 
	+\sum_{j=i,e} \int_{\Oj} 
	\abs{\nabla u_j^{\e}}^2\, dx\right).
\end{equation}

Insert \eqref{ionL2}, \eqref{source}, \eqref{source-part} 
into \eqref{vL2}, use \eqref{eq:matrix2}, and choose $\delta$ 
(appropriately) small. Integrating the result over the time 
intervall $(0,\tau)$, for $\tau\in (0,T)$, yields
\begin{align*}
	& \e \left( \norm{v_{\e}(\tau)}^2_{L^2(\Ge)} 
	+ \norm{w_{\e}(\tau)}^2_{L^2(\Ge)}\right)
	+ \frac{\alpha}{2} \sum_{j=i,e}\int_0^\tau 
	\norm{\nabla u_j^{\e}}^2_{L^2(\Oj)} \,dt 
	\\ & \quad \quad 
	+ \gamma \e \int_0^\tau \norm{v^{\e}(t)}^4_{L^4(\Omega)} \,dt 
	\leq \e \left( \norm{v_{\e}(0)}^2_{L^2(\Ge)} 
	+ \norm{w_{\e}(0)}^2_{L^2(\Ge)}\right)	
	\\ &  \quad\quad\quad
	+ C_2 \e \int_0^\tau \left( \norm{v_{\e}(t)}^2_{L^2(\Ge)} 
	+ \norm{w_{\e}(t)}^2_{L^2(\Ge)}\right) \, dt  
	+ C_\delta \sum_{j=i,e}\norm{s_j^{\e}}^2_{L^2((0,\tau)\times \Oj)},
\end{align*}
for some positive constant $C_2$ independent of $\e$. 
Applying Gr\"{o}wall's inequality, recalling 
\eqref{assumptionInitial} and \eqref{assumptionStimulation}, we obtain 
estimates (a), (c), (d), (e) in \eqref{eq:main-est}. 
Estimate (b) follows from \eqref{source-part} and (a), (c). 
Finally, note that \eqref{GFHN} implies the bound
$$
\abs{\dt w^\e}^2 \le C 
\left(1+\abs{v^\e}^4 + \abs{v^\e}^2\right).
$$
Estimate (f) follows by integrating this over 
$(0,T)\times \Ge$ and using (d), (e).

\end{proof}

\begin{remark}\label{rem:motivate-regularity}
Let us motivate the regularity requirements in 
Definition \ref{weak} that are not covered 
by Lemma \ref{lemma:energy}. 
First, due to \eqref{ve-def}, \eqref{trace1} 
and \eqref{eq:main-est},
$$
\norm{v^{\e}}_{L^2(0,T;H^{1/2}(\Ge))}
\leq \tilde C_{\e} 
\sum_{j=i,e}\norm{u_j^{\e}}_{L^2(0,T;H^1(\Oj))}
\le C_\e,
$$
where the constants $\tilde C_\e, C_{\e}$ may depend on $\e$ (via 
\eqref{trace1} with $\G$ replaced by $\Ge$). 

Next, we claim that 
$$
\e\norm{\dt v^\e}_{L^2(0,T;H^{-1/2}(\Ge))+L^4((0,T)\times \Ge)}\le C_\e,
$$
for some constant that may depend on $\e$. To see this use a version 
of \eqref{weak_i} or \eqref{weak_e} 
(with time-independent test functions) to write 
$$
\e \left\langle \dt v^\e , \psi\right\rangle_{H^{-1/2}(\Ge),H^{1/2}(\Ge)}
=J_1(\psi) +J_2(\psi)+J_3(\psi),
$$
for a.e.~$t\in (0,T)$ and for any $\psi\in H^{1/2}(\Ge)$ with 
$\norm{\psi}_{H^{1/2}(\Ge)}=1$, where
\begin{align*}
	&\abs{J_1(\psi)}=\abs{\int_{\Oj}\sigma_j^{\e} \nabla u_j^{\e} 
	\cdot \nabla \mathcal{I}_j(\psi)\, dx}, 
	\qquad \abs{J_2(\psi)} = 
	\abs{\int_{\Oj}s^{\e}_j\, \mathcal{I}_j(\psi) \, dx},
	\\ & \abs{J_3(\psi)}
	=\abs{\e \int_{\Ge} I(v^{\e},w^{\e})\psi \,dS},
\end{align*}
and $\mathcal{I}_j(\cdot)$ is the right inverse of the 
trace operator relative to $\Oj$, for $j=i$ or $j=e$. 
Clearly, using the Cauchy-Schwarz inequality 
and \eqref{inverseTrace},
\begin{align*}
	& \abs{J_1(\psi)}\le \tilde C_1 
	\norm{\nabla u_j^{\e}}_{L^2(\Oj)}
	\norm{\psi}_{H^{1/2}(\Ge)},
	\quad
	\abs{J_2(\psi)}\le \tilde C_2 
	\norm{s_j^{\e}}_{L^2(\Oj)}
	\norm{\psi}_{H^{1/2}(\Ge)},
	\\ & \text{and} \quad 
	\abs{J_3(\psi)}\le 
	\e \norm{I(v^{\e},w^{\e})}_{H^{-1/2}(\Ge)}
	\norm{\psi}_{H^{1/2}(\Ge)},
\end{align*}
where the constants $\tilde C_1, \tilde C_2$ may depend on 
$\e$ (via the inverse trace inequality \eqref{inverseTrace} with $\G$ 
replaced by $\Ge$). As a result, for a.e.~$t\in (0,T)$,
$$
\norm{J_1}_{H^{-1/2}(\Ge)}^2
\le \tilde C_1^2\norm{\nabla u_j^{\e}}_{L^2(\Oj)}^2,
\quad
\norm{J_2}_{H^{-1/2}(\Ge)}^2
\le \tilde C_2^2\norm{s_j^{\e}}_{L^2(\Oj)}^2.
$$
Integrating this over $t\in (0,T)$ and using 
\eqref{eq:main-est}-(a), \eqref{assumptionStimulation} 
it follows that
$$
\norm{J_1}_{L^2(0,T;H^{-1/2}(\Ge))}\le C_1,
\qquad 
\norm{J_2}_{L^2(0,T;H^{-1/2}(\Ge))}\le C_2,
$$
for some constants $C_1, C_2$ (that may depend on $\e$).
 
As in Remark \ref{weak:welldef}, 
$$
\norm{I(v^{\e},w^{\e})}_{H^{-1/2}(\Ge)}
\le C_I \left(1+\norm{v^{\e}}_{L^4(\Ge)}^4
+\norm{w^{\e}}_{L^2(\Ge)}^2\right)^{3/4},
$$
and hence, for a.e.~$t\in (0,T)$,
$$
\norm{J_3}_{H^{-1/2}(\Ge)}^{4/3}
\le \tilde C_3 \e^{4/3}\left(1+\norm{v^{\e}}_{L^4(\Ge)}^4
+\norm{w^{\e}}_{L^2(\Ge)}^2\right).
$$
Integrating over $t\in (0,T)$ and using 
estimates \eqref{eq:main-est}-(d,e), we conclude
$$
\norm{J_3}_{L^{4/3}(0,T;H^{-1/2}(\Ge))}\le C_3,
$$
for some constant $C_3$ (independent of $\e$).

\end{remark}

For the upcoming convergence analysis, we need  
a temporal translation estimate for the membrane potential $v^\e$.

\begin{lemma}[temporal translation estimate in $L^2$]\label{lemma:temp-shift}
Referring back to Theorem \ref{thm:micro}, let  
$(u_i^{\e},u_e^{\e},v^{\e},w^{\e})$ be the weak solution of 
\eqref{micro}, \eqref{eq:ib-cond}. Then $v^\e$ 
satisfies the following $L^2$ temporal 
translation estimate for some $\e$-independent 
constant $C>0$, for any sufficiently small $\Delta_t>0$:
\begin{equation}\label{eq:ve-time-trans}
	\e \int_0^{T-\Delta_t}\int_{\Ge} 
	\abs{v^\e(t+\Delta_t,x)-v^\e(t,x)}^2 \,dx \,dt 
	\le C \Delta_t.
\end{equation}
\end{lemma} 

\begin{proof} 
The translated functions $(u_i^{\e},u_e^{\e},v^{\e},w^{\e})(t+\Delta_t)$
constitute a weak solution of \eqref{micro} on $(0,T-\Delta_t)$ with initial 
data $(v^\e,w^\e)(\Delta_t)$ and stimulation 
currents $(s^\e_i,s^\e_e)(t+\Delta_t)$. 
We subtract the equations (linked to $\Oi$ and $\Oe$) 
for the original weak solution from the equations satisfied by 
the translated one, and add the resulting equations. The result is
\vspace{-4.0em}
\begin{equation*}\label{eq-translation_1}
	\begin{split}
		&-\e \int_0^{T-\Delta_t} \int_{\Ge} 
		\left(v^{\e}(t+\Delta_t)-v^{\e}(t) \right)
		\dt (\p_i-\p_e) \, dS\, dt
		\\ & \quad \quad
		+\e \int_{\Ge} \left(v^{\e}(T)-v^{\e}(T-\Delta_t) \right) (\p_i(T-\Delta_t)-\p_e(T-\Delta_t)) \,dS		
		\\ & \quad \quad
		-\e \int_{\Ge} \left(v^{\e}(\Delta_t)-v^{\e}_0 \right) (\p_i(0)-\p_e(0)) \,dS
		\\ & \quad\quad 
		+ \sum_{j=i,e} \int_0^{T-\Delta_t} \int_{\Oj} 
		\sigma_j^{\e} \nabla \left(u_j^{\e}(t+\Delta_t)-u_j^{\e}(t)\right) 
		\cdot \nabla \p_j \, dx\, dt
		\\ & \quad\quad  
		+\e\int_0^{T-\Delta_t} \int_{\Ge} 
		\left( I(v^{\e}(t+\Delta_t),w^{\e}(t+\Delta_t)
		- I(v^{\e}(t),w^{\e}(t)\right) (\p_i-\p_e) \,dS\, dt
		\\ & \quad\quad
		-\e \int_0^{T-\Delta_t}\int_{\Ge} 
		\left(H(v^{\e}(t+\Delta_t),w^{\e}(t+\Delta_t))
		-H(v^{\e}(t),w^{\e}(t)\right) \p_w \,dS\, dt 
		\\ & \quad \quad\qquad			
		=\sum_{j=i,e}\int_0^{T-\Delta_t} \int_{\Oj} 
		\left(s_j^{\e}(t+\Delta_t)-s^{\e}(t) \right) \p_j \, dx\, dt.
	\end{split}
\end{equation*}
Specifying the test functions $\p_i, \p_e,\p_w$ as
$$
\p_j(t)= -\int_t^{t+\Delta_t} u_j^\e(s) \,ds,
\quad j=i,e,
\qquad
\p_w(t)= -\int_t^{t+\Delta_t} w_j^\e(s) \,ds,
$$
we obtain			
\begin{equation*}
	\begin{split}
		& \e \int_0^{T-\Delta_t} 
		\int_{\Ge} \abs{v^{\e}(t+\Delta_t)-v^{\e}(t)}^2\, dS \, dt 
		\\ & \quad\quad 
		+ \e \int_{\Ge} \left(v^{\e}(T)-v^{\e}(T-\Delta_t) \right)  
		\left(-\int_{T-\Delta_t}^T v(s)\,ds \right)\, dS 
		\\ & \quad\quad 
		- \e \int_{\Ge} \left(v^{\e}(\Delta_t)-v^{\e}_0 \right)  
		\left(-\int_0^{\Delta_t} v(s)\,ds \right)\, dS 
		\\ & \quad \quad
		+\sum_{j=i,e}\int_0^{T-\Delta_t} \int_{\Oj}
		\sigma_j^{\e} \nabla \left(u_j^{\e}(t+\Delta_t)-u_j^{\e}(t)\right) 
		\\& \quad \qquad \qquad\qquad\qquad\qquad\quad
		\cdot \nabla \left(-\int_{t}^{t+\Delta_t}u_j^{\e}(s)\,ds\right) \, dx\, dt
		\\ & \quad\quad  
		+ \e \int_0^{T-\Delta_t}\int_{\Ge}
		\left( I(v^{\e}(t+\Delta_t),w^{\e}(t+\Delta_t))
		- I(v^{\e}(t),w^{\e}(t)) \right)
		\\ & \quad \qquad \qquad\qquad\qquad\qquad\quad
		 \times \left(-\int_{t}^{t+\Delta_t}v^{\e}(s)\,ds\right) \,dS \, dt
		\\ & \quad\quad
		+ \e \int_0^{T-\Delta_t}\int_{\Ge}
		\left( H(v^{\e}(t+\Delta_t),w^{\e}(t+\Delta_t))
		- H(v^{\e}(t),w^{\e}(t))\right)
		\\ & \quad \qquad \qquad\qquad\qquad\qquad\quad
		\times \left(\int_{t}^{t+\Delta_t}w^{\e}(s)\,ds\right) \,dS \, dt 
		\\ & \qquad \quad \qquad
		= \sum_{j=i,e}\int_0^{T-\Delta_t} \int_{\Oj} 
		s_j^\e \left(-\int_{t}^{t+\Delta_t}u_j^{\e}(s)\,ds\right) \, dx \, dt.
	\end{split}
\end{equation*} 	
Let us write this equation as 
$$
\e \int_0^{T-\Delta_t} \int_{\Ge}
\abs{v^{\e}(t+\Delta_t)-v^{\e}(t)}^2\, dS\, dt
+J_1+J_2+J_3+J_4+J_5=J_6. 
$$

By the Cauchy-Schwarz and Minkowski integral inequalities,
\begin{equation*}
	\begin{split}
		&\abs{J_1}
		\leq \e\norm{v^{\e}(T)-v^{\e}(T-\Delta_t)}_{L^2(\Ge)}
		\norm{\int_{T-\Delta_t}^T v^{\e}(s)\,ds}_{L^2(\Ge)}
		\\ & \qquad
		\leq 2 \e \norm{v^{\e}}_{L^\infty(0,T;L^2(\Ge))} 
		\int_{T-\Delta_t}^T\norm{v^{\e}}_{L^2(\Ge)}\,ds	
		\\ & \qquad
		\leq 2 \e \norm{v^{\e}}_{L^{\infty}(0,T;L^2(\Ge))}^2\Delta_t
		\overset{\eqref{eq:main-est}-(c)}{\le} C_1 \Delta_t.
	\end{split}	 	
\end{equation*}
Similarly, $\abs{J_2}\leq C_2\Delta_t$.

We need the following facts 
involving a real-valued function $f\in L^p$ ($p>1$):
\begin{equation} \label{convolution}
	\begin{split}
		& \int_t^{t+\Delta_t} f(s)\,ds 
		= \left(\mathbbm{1}_{[0,\Delta_t]} \star f\right)(t), 	
		\\ & 
		\norm{\mathbbm{1}_{[0,\Delta_t]} \star f}_{L^p} 
		\leq \norm{\mathbbm{1}_{[0,\Delta_t]}}_{L^1}
		\norm{f}_{L^p}=\Delta_t \norm{f}_{L^p},
	\end{split}
\end{equation}
where the second line follows from Young's convolution inequality. 

By \eqref{eq:matrix1} and the Cauchy-Schwarz inequality, 
\begin{align*}	
	& \abs{J_3} \le \sum_{j=i,e}
	2 \norm{\sigma_j^\e}_{L^\infty}
	\norm{\nabla u_j^{\e}}_{L^2((0,T)\times \Oj)} 
	\norm{\int_{t}^{t+\Delta_t} 
	\nabla u_j^{\e}(s)\,ds}_{L^2((0,T-\Delta_t)\times \Oj)}.
\end{align*}
As a result of Minkowski integral inequality,
\begin{align*}
	& \norm{\int_{t}^{t+\Delta_t} \nabla u_j^{\e}(s)
	\,ds}_{L^2((0,T-\Delta_t)\times \Oj)}^2
	\\ & \qquad = \int_0^{T-\Delta_t}\int_{\Oj} 
	\abs{\int_t^{t+\Delta_t}\nabla u_j^\e(s)
	\, ds}^2\, dx \,dt
	\\ & \qquad 
	\le \int_0^{T-\Delta_t} \left(\int_t^{t+\Delta_t} 
	\left(\int_{\Oj} \abs{\nabla u_j^\e(s)}^2 
	\,dx\right)^{1/2} \, ds\right)^2 \,dt
	\\ & \qquad 
	=  \int_0^{T-\Delta_t}\left(\int_t^{t+\Delta_t} 
	\norm{\nabla u_j^\e(s)}_{L^2(\Oj)} \, ds\right)^2 \,dt 
	\\ & \qquad
	\overset{\eqref{convolution}}{\le}
	\Delta_t^2 \int_0^{T-\Delta_t} 
	\norm{\nabla u_j^\e}_{L^2(\Oj)}^2 \, dt
	\le \Delta_t^2 \norm{\nabla u_j^\e}_{L^2((0,T)\times \Oj)}^2.
\end{align*}
Therefore,
$$
\abs{J_3}\le \sum_{j=i,e}
2 \norm{\sigma_j^\e}_{L^\infty}
\norm{\nabla u_j^\e(t)}_{L^2((0,T)\times \Oj)}^2\Delta_t
\overset{\eqref{eq:main-est}-(a)}{\le} 
C_3\Delta_t.
$$

An application of H\"{o}lder's inequality yields 
\begin{equation*}
	\begin{split}
		\abs{J_4} &
		\leq 2 \e \norm{I^{\e}}_{L^{4/3}((0,T)\times \Ge)}
		\norm{\int_{t}^{t+\Delta_t}v^{\e}(s) \,ds}_{L^4((0,T-\Delta_t)\times \Ge)}
		\\ & 
		\overset{\eqref{I43}}{\le} C_4 
		\e \left(\norm{v^{\e}}^4_{L^4((0,T)\times \Ge)} 
		+ \norm{w^{\e}}_{L^2((0,T)\times \Ge)}^2\right)^{3/4} 
		\Delta_t \norm{v^{\e}}_{L^4((0,T)\times \Ge)}
		\\&
		\le C_5 \Delta_t \e \left( \norm{v^{\e}}_{L^4((0,T)\times \Ge)}^4 
		+ \norm{w^{\e}}_{L^2((0,T)\times \Ge)}^2 \right)
		\overset{\eqref{eq:main-est}-(d,e)}{\le} C_6 \Delta_t,
	\end{split}	
\end{equation*}
where we have repeated the argument for $J_1$ 
involving \eqref{convolution}, with the implication
that $\norm{\int_{t}^{t+\Delta_t}v^{\e}(s)\,ds}_{L^4_{t,x}}$ 
is bounded in terms of $\Delta_t \norm{v^{\e}}_{L^4_{t,x}}$. 
Moreover, we have used the basic inequality 
$ab\le \frac34 a^{4/3} + \frac14 b^4$ for positive numbers $a,b$.

Similarly,
\begin{equation*}
	\begin{split}
		\abs{J_5} 
		& \overset{\eqref{GFHN}}{\le} 
		C_7\Delta_t \e \left( \norm{v^{\e}}_{L^4((0,T)\times \Ge)}^4
		+\norm{w^{\e}}_{L^2((0,T)\times \Ge)}^2\right)^{1/2}
		\norm{v^\e}_{L^2((0,T)\times \Ge)}
		\\ & \overset{\eqref{eq:main-est}-(d,e)}{\le} C_8 \Delta_t,
	\end{split}
\end{equation*}
and
\begin{equation*}
	\abs{J_6}\leq 2 \Delta_t 
	\sum_{j=i,e}\norm{s_j^{\e}}_{L^2((0,T)\times \Oj)}
	\norm{u_j^{\e}}_{L^2((0,T)\times \Oj)}
	\overset{\eqref{assumptionStimulation}, \eqref{eq:main-est}-(b)}{\le}
	C_9 \Delta_t.
\end{equation*}

Summarizing our findings, we conclude that \eqref{eq:ve-time-trans} holds.
\end{proof}

\begin{remark}
Due to the degenerate structure of the 
microscopic system \eqref{micro}, temporal 
estimates are not available for 
the intra- and extracellular potentials $u_i^\e, u_e^\e$. 
\end{remark}

\section{The homogenization result} \label{sec:convergence}
This section contains the main result of the paper. 
We start by recalling the weak formulation of the 
macroscopic bidomain system \eqref{macro}, which 
is augmented with the following initial and boundary conditions:
\begin{equation}\label{eq:ib-cond2}
	\begin{split}
		& v|_{t=0}=v_0 \;\;  \text{in $\O$}, \quad
		w|_{t=0}=w_0, \;\;  \text{in $\O$}.
		\\ & n \cdot\sigma  \nabla u_j  = 0  
		\;\;  \text{on $(0,T)\times \partial \O$, $j=i,e$.}
	\end{split}
\end{equation}

\begin{definition}[weak formulation of bidomain system]\label{weakMacro}
A weak solution to \eqref{macro}, \eqref{eq:ib-cond2} is a 
collection $(u_i,u_e,v,w)$ of functions 
satisfying the following conditions:
\begin{enumerate}[i]
	\item (algebraic relation). 
	$$
	v= u_i - u_e \quad \text{a.e.~in $(0,T)\times\O$}. 
	$$

	\item (regularity).
	\begin{align*}
		&u_j \in L^2(0,T;H^1(\O)), \quad j=i,e, \\
		& \int_{\O} u_e^\e(t,x)\, dx = 0, \quad t\in (0,T), \\
		&v \in L^2(0,T; H^1(\O)) \cap L^4((0,T)\times \O), \\
		& \dt v \in L^2(0,T; H^{-1}(\O)) + L^{4/3}((0,T)\times \O), \\
		& w \in H^1(0,T;L^2(\O)).
	\end{align*}

	\item (initial conditions).
 	\begin{align*}
		v(0)=v_0, \quad
		w(0) = w_0.
	\end{align*} 

	\item (differential equations).
	\begin{equation*}
		\begin{split}
			&|\G|\int_0^T \left\langle \dt v,  \p_i 
			\right\rangle  \,dt
			+ \int_0^T \int_{\O} M_i \nabla u_i \cdot \nabla \p_i\, dx \, dt
			\\ & \qquad\qquad
			+  |\G| \int_0^T \int_{\O} I(v,w)\p_i \, dx \, dt 
			= |Y_i|\int_0^T\int_{\O} s_{i} \p_i \, dx \, dt,
		\end{split}
	\end{equation*}
	
	\begin{equation*}
		\begin{split}
			& |\G| \int_0^T \left\langle \dt v,\p_e
			\right\rangle  \,dt
			- \int_0^T \int_{\O} M_e \nabla u_e\cdot \nabla \p_e \, dx \, dt
			\\ & \qquad\qquad
			+ |\G|\int_0^T \int_{\O} I(v,w)\p_e \, dx \, dt
			= -|Y_e|\int_0^T \int_{\O} s_{e} \p_e \, dx \, dt,
		\end{split}
	\end{equation*}
	
	\begin{equation}
		\int_0^T \int_{\O} \dt w \p_w \, dx \, dt 
		=  \int_0^T \int_{\O} H(v,w) \p_w \, dx \, dt,
	\end{equation}
\end{enumerate}
for all test functions $\p_i,\p_e\in L^2(0,T; H^1(\O))\cap L^4((0,T)\times \O)$, 
$\p_w\in L^2(0,T;L^2(\O))$. We denote by $\left\langle \cdot, \cdot \right\rangle$ 
the duality pairing between $L^2(0,T; H^{-1}(\O)) + L^{4/3}((0,T)\times \O)$ 
and $L^2(0,T; H^1(\O))\cap L^4((0,T)\times \O)$.
\end{definition} 

The macroscopic bidomain system is well studied for a variety of cellular 
models \cite{Andreianov:2010uq,Karlsen,Boulakia2008,Bourgault,Colli,Veneroni:2009aa}. 
For the following result, see \cite{Andreianov:2010uq,Karlsen,Boulakia2008}.

\begin{theorem}[well-posedness of bidomain system]
Suppose \eqref{def:domain} holds, $I$ and 
$H$ satisfy the conditions in \eqref{GFHN} 
and \eqref{uniq-cond1} (for uniqueness), $M_i(x)$ and $M_e(x)$ 
are bounded, positive definite matrices, 
$s_i,s_e\in L^2((0,T)\times \O)$, and $v_0,w_0 \in L^2(\O)$. 
Then there exists a unique weak solution 
to the bidomain system \eqref{macro}, \eqref{eq:ib-cond2}.
\end{theorem}

We are now in a position to state the 
main result, which should be compared to 
Theorem 1.3 in \cite{Pennacchio2005}.

\begin{theorem}[convergence to the bidomain system]\label{theorem:homo}
Suppose conditions \eqref{GFHN}, \eqref{def:domain}, \eqref{uniq-cond1}, 
\eqref{fastSlow}, \eqref{eq:matrix1}, \eqref{eq:matrix2}, 
\eqref{assumptionCompatibility}, \eqref{sj-ass}, and \eqref{assumptionInitial} hold. 
Let $\e$ take values in a sequence tending to zero (e.g.~$\e^{-1}\in \N$). 
Then the sequence $\seq{u_i^{\e},u_e^{\e},v^{\e},w^{\e}}_{\e>0}$ 
of weak solutions to the microscopic system \eqref{micro}, \eqref{eq:ib-cond} 
two-scale converges (in the sense of \eqref{convergence} below)
to the weak solution $(u_i,u_e,v,w)$ of the 
macroscopic bidomain system \eqref{macro}, \eqref{eq:ib-cond2}. 
Moreover, $\seq{v^{\e}}_{\e>0}$ converges strongly in the sense that
\begin{equation}\label{strongConvergence}
	\e^{1/2} \norm{v^{\e}-v}_{L^2\left( \Ge \right)} \to 0, 
	\quad \text{as $\e \rightarrow 0$.}
\end{equation} 
\end{theorem}  

\begin{remark}
The strong convergence \eqref{strongConvergence} is a 
corrector-type result. In the current setting it allows us to pass 
to the limit in the nonlinear ionic terms. By employing standard 
techniques \cite{Donato} one can also show that the energies 
$$
\seq{\int_0^T\int_{\Oj} \sigma_j^{\e} \nabla u_i^{\e}
\cdot  \nabla u_j^{\e} \, dx \, dt}_{\e>0}
$$
converge to the homogenized energy  
$$
\int_0^T\int_{\O} M_j \nabla u_i
\cdot  \nabla u_j \, dx \, dt,
$$
where $M_j$ is defined in \eqref{Mj}, $j=i,e$.
\end{remark}

The rest of this section is devoted to the proof 
of Theorem \ref{theorem:homo}. Homogenization of
the linear terms in \eqref{micro} is handled with 
standard techniques, cf.~Subsection \ref{subsec:two-scale}.

Passing to the limit in the nonlinear terms $I,H$ 
is more challenging. Although the intracellular/extracellular functions 
$u_i^{\e}$ (defined on $\O_i^\e$) and $u_e^{\e}$ (defined on $\O_e^\e$) 
do not converge strongly, some kind of strong compactness 
is expected for the $\e$-scaled version of the transmembrane potential 
$v^{\e}= u_i^{\e}\big|_{\Ge}-u_e^{\e}\big|_{\Ge}$ (defined on $\Ge$), since 
we control both the temporal \eqref{eq:ve-time-trans} and 
spatial (fractional) derivatives \eqref{eq:main-est}.

Wild oscillations of the underlying domain 
do however pose difficulties. For this reason, we use the 
boundary unfolding operator $\Tau_{\e}^b$, cf.~\eqref{unfolding}, 
to transform the problem of convergence on the oscillating set 
$\Ge$ to the fixed set $\O \times \G$.  Roughly speaking, \eqref{eq:main-est} 
is used to conclude that $\Tau_{\e}^b(v^\e)$ is uniformly 
bounded in $L^2_tL^2_xH^{1/2}_y$. In addition, in view 
of \eqref{eq:ve-time-trans}, $\Tau_{\e}^b(v^\e)$ possesses an $\e$-uniform 
temporal translation estimate with respect to the $L^2_{t,x,y}$ norm.
As a result, the Simon compactness result (cf.~Subsection  \ref{sec:notation}) 
implies that $\seq{(t,y)\mapsto \Tau_{\e}^b(v^\e)(t,x,y)}_{\e>0}$ is 
precompact in $L^2((0,T)\times \G)$, for fixed $x$. 
Next we demonstrate that $\Tau_{\e}^b(v^\e)$ is equicontinuous in $x$ (with values
in $L^2((0,T)\times \G)$). Applying the Simon-type compactness criterion 
found in \cite{Neuss} (cf.~Theorem \ref{compactSurface0} below), it follows that 
$\seq{\Tau_{\e}^b(v^{\e})}_{\e>0}$ converges along a subsequence. 
Owing to the uniqueness of solutions to the bidomain system \eqref{weakMacro}, 
the entire sequence $\seq{\Tau_{\e}^b(v^{\e})}_{\e>0}$ 
converges (not just a subsequence). We refer to Subsection \ref{subsec:strongconv} 
for details. For inspirational works deriving macroscopic 
models by combining two-scale and unfolding techniques, we refer to 
\cite{Neuss,Gahn:2016aa,Graf:2014aa,Marciniak-Czochra:2008aa,Neuss-Radu:2007aa}.

\subsection{Extracting two-scale limits}\label{subsec:two-scale}
Recall that $\tilde{\cdot}$ denotes the extension to $\Omega$ by 
zero, and that $\mathbbm{1}_{Y_j}$ is the indicator function of $Y_j$ ($j=i,e$). 
Using the a priori estimates provided by Lemma \ref{lemma:energy}, we 
can apply Lemma \ref{twoH1perf}, Theorem \ref{twoS}, 
and Lemma \ref{relation trace} to extract 
subsequences (not relabelled) such that
\begin{equation}\label{convergence}
	\begin{split}
		&\widetilde{u_j^{\e}} 
		\overset{2}{\rightharpoonup} 
		\mathbbm{1}_{Y_j}(y)u_j(t,x), \quad j=i,e
		\\ &
		\widetilde{\nabla u_j^{\e}} \overset{2}{\rightharpoonup}
		\mathbbm{1}_{Y_j}(y)\left( \nabla_x u_j(t,x) + \nabla_y u_j^1(t,x,y) \right),
		\quad j=i,e,
		\\ &
		v^{\e}\overset{2-\mathrm{S}}{\rightharpoonup}  v = u_i-u_e, 
		\\ &
		w^{\e}\overset{2-\mathrm{S}}{\rightharpoonup} w, 
		\quad 
		\dt w^{\e} \overset{2-\mathrm{S}}{\rightharpoonup} \dt w, 
		\\ & 
		I(v^{\e},w^{\e}) \overset{2-\mathrm{S}}{\rightharpoonup}  \overline{I}, 
		\quad
		H(v^{\e},w^{\e}) \overset{2-\mathrm{S}}{\rightharpoonup} \overline{H},
	\end{split}
\end{equation}
for some limits $u_i,u_e \in L^2(0,T,H^1(\O))$, 
$u_i^1,u_e^1 \in L^2((0,T)\times \O ;H_{\mathrm{per}}^1(Y))$, and 
$w \in L^2((0,T)\times\O \times \G )$. Here we identify 
$v=u_i-u_e$ as an element in $L^2(0,T,H^1(\O))$. It is easily 
verified that the two-scale limit 
$u_e$ satisfies $\int_{\O} u_e \, dx=0$. 

Nonlinear functions are not continuous with respect to weak 
convergence, which prevents us from immediately 
making the identifications 
$$
\overline{I} = I(v,w), \qquad 
\overline{H} = H(v,w).
$$  

Using the two-scale convergences in \eqref{convergence}, 
Remark \ref{rem:well-defined}, the choice of test function 
\begin{align*}
	& \p(t,x)+\e\p_1\left(t,x,\frac{x}{\e}\right), 
	\quad \text{with}
	\\
	& \p \in C^{\infty}_0((0,T)\times \O),
	\; 
	\p_1 \in C^\infty_0((0,T)\times\O;C^\infty_{\mathrm{per}}(Y)),
\end{align*}
in \eqref{weak_i}, \eqref{weak_e}, and \eqref{weak_w}, standard 
manipulations \cite{Donato,Allaire2} will reveal 
that the two-scale limit $(u_i,u_e,v,w)$ satisfies 
the following equations:
\begin{equation}\label{u_iLim}
	\begin{split}
		&|\Gamma|\int_0^T \int_{\O} 
		\left\langle \dt v, \p \right\rangle \, dx \, dt 
		+\lim_{\e\rightarrow 0}\e \int_0^T 
		\int_{\G^{\e}} I(v^{\e},w^{\e}) \p \,dS \,dt
		\\ & \qquad\qquad\qquad
		+\int_0^T\int_{\O} M_i \nabla u_i\cdot  \nabla \p \, dx \, dt 
		=|Y_i|\int_0^T\int_{\O} s_i \p \, dx \, dt,
	\end{split}
\end{equation}
\begin{equation}\label{u_eLim}
	\begin{split}
		& |\Gamma|\int_0^T \int_{\O} 
		\left\langle \dt v, \p \right\rangle \, dx \, dt 
		+\lim_{\e\rightarrow 0} \e \int_0^T \int_{\G^{\e}} I(v^{\e},w^{\e}) 
		\p \,dS \,dt
		\\ & \qquad\qquad\qquad 
		+\int_0^T\int_{\O} M_e \nabla u_e\cdot  \nabla \p \, dx \, dt 
		= |Y_e|\int_0^T\int_{\O} s_e \p\, dx \, dt,
	\end{split}
\end{equation}  
and 
\begin{equation}\label{wLim}
	|\G|\int_0^T \int_{\O} \dt w \p \, dx \, dt 
	=  \lim_{\e\rightarrow 0} \e \int_0^T 
	\int_{\G^{\e}} H(v^{\e},w^{\e}) \p \,dS \,dt,
\end{equation}
for all $\p \in C^{\infty}_0((0,T)\times \O)$. 

In \eqref{u_iLim} and \eqref{u_eLim}, $M_i$ 
and $M_e$ are the homogenized conductivity tensors $\eqref{Mj}$. 
Let us briefly recall how one arrives at the 
homogenized conductivities. Setting $\p\equiv 0$ and considering 
$$
\Phi(t,x,y) := \sigma_j(x,y)\nabla_y \p_1(t,x,y)
$$ 
as a test function for two-scale convergence, we have 
by \eqref{convergence} that
\begin{align*}
	&\lim_{\e\to 0} \int_0^T  \int_{\Oj} 
	\Phi\left(t,x,\frac{x}{\e}\right) \cdot
	\nabla u_j^{\e} \, dx\,dt 
	\\ & \qquad 
	=\int_0^T \int_{\O} \int_{Y_j} \sigma_j 
	\left[ \nabla_x u_j + \nabla_y u^1_j \right] 
	\cdot  \nabla_y \p_1 \,dx\, dy\, dt. 
\end{align*}
Note that the oscillating $\p_1$ term is suppressed in the limit of 
the weak formulation \eqref{weak} by the $\e$-factor, except 
in the term where a gradient hits the test function. 
Thus, the two-scale limit $(u_j,u_j^1)$ 
satisfies the equation
$$
\int_0^T \int_{\O}\int_{Y_j} 
\sigma_j \left[ \nabla_x u_j + \nabla_y u^1_j \right] 
\cdot  \nabla_y \p_1 \,dx \, dy\,dt =0, 
$$
for all $\p_1 \in C_0^{\infty}((0,T)\times\O;C_{\mathrm{per}}^{\infty}(Y))$.
This equation is satisfied by $u_j^1 = \chi_j \cdot \nabla_x u_j$, where 
$\chi_j$ is the first order corrector \eqref{chi}. Hence, for any 
$y$-independent function $\p_1(t,x,y)\equiv 
\psi(t,x) \in C_0^{\infty}((0,T)\times \O)$,
\begin{align*}
	&0=\int_0^T \int_{\O}\int_{Y_j} 
	\sigma_j \left[ \nabla_x u_j + \nabla_y u^1_j \right] 
	\cdot \nabla_x \psi \,dx \, dy\,dt
	\\ & \qquad = \int_0^T\int_{\O}  
	\left(\int_{Y_j} \sigma_j  + \nabla_y \chi_j \, dy \right) 
	\nabla_x u_j \cdot \nabla_x \psi \,dx\,dy\, dt,
\end{align*}
so \eqref{Mj} is indeed the homogenized conductivity tensor.
  
Up to the convergences of the nonlinear terms, it is now clear 
that \eqref{u_iLim}, \eqref{u_eLim}, and \eqref{wLim} 
constitute the weak formulation of 
the bidomain system  \eqref{macro}, \eqref{eq:ib-cond2}  
(in the sense of Definition \ref{weakMacro}).  

\subsection{The nonlinear terms and strong convergence}\label{subsec:strongconv}
To finalize the proof of Theorem \ref{theorem:homo}, 
it remains to identify the limits
\begin{equation}\label{limitIntegrals}
	\begin{split}
		\lim_{\e\rightarrow 0} 
		\e \int_0^T  \int_{\G^{\e}} I(v^{\e},w^{\e}) \p \,dS \,dt
		& = |\G| \int_0^T \int_{\O}  I(v,w) \p \, dx \, dt, \\
		 \lim_{\e\rightarrow 0} 
		 \e \int_0^T  \int_{\G^{\e}} H(v^{\e},w^{\e}) \p \,dS \,dt
		 & =  |\G| \int_0^T  \int_{\O}  H(v,w) \p \, dx \, dt,
	\end{split}
\end{equation}
for all $\p \in C^{\infty}_0((0,T)\times\O)$.  
We note that the unfolding operator \eqref{IntegralTransform} allows 
us to transform the oscillating surface integral
$$
\e \int_0^T \int_{\Ge} I(v^{\e},w^{\e})\p_i \,dS \, dt
$$
into 
$$
\int_0^T  \int_{\Omega}\int_{\G} 
I\left(\Tau_{\e}^b(v^{\e}),\Tau_{\e}^b(w^{\e})\right) 
\Tau_{\e}^b(\p_i) \,dS(y) \, dx\, dt,
$$
coming from the integration formula \eqref{IntegralTransform} 
for $\Tau_{\e}^b$. Additionally, we have here used \eqref{Tb-product} 
and the fact $\Tau_{\e}^b \left(I(v^{\e},w^{\e})\right)
=I\left(\Tau_{\e}^b(v^{\e}),\Tau_{\e}^b(w^{\e})\right)$. 
The smoothness of $\p$ implies that $\Tau_{\e}^b(\p)\to \p$ 
in $L^2((0,T)\times\O \times \G )$ as $\e\to 0$, 
cf.~\eqref{Tb-strongconv}, so to identify the 
limits \eqref{limitIntegrals} it suffices to show
$$
I\left(\Tau_{\e}^b(v^{\e}),\Tau_{\e}^b(w^{\e})\right) \to I(v,w), 
\quad  \text{weakly in $L^2((0,T)\times\O \times \G)$},
$$
where $v,w$ are identified in \eqref{convergence}. 
Besides, since $w^{\e}$ appears linearly in 
$I$ and $H$, cf.~\eqref{GFHN}, the weak convergence 
of $\seq{\Tau_{\e}^b(w^{\e})}_{\e>0}$ in 
$L^2((0,T)\times\O\times\G)$ is enough to pass to the 
limit in \eqref{limitIntegrals}, if we establish 
strong convergence of $\seq{\Tau_{\e}^b(v^{\e})}_{\e>0}$.

As a step towards verifying the required strong convergence, we 
need to show that $\seq{(t,y)\mapsto \Tau_{\e}^b (v^\e)}_{\e>0}$ 
is strongly precompact in $L^2_{t,y}$, for fixed $x\in \O$. 
As a result of Lemma \ref{lemma:energy}, $\Tau_{\e}^b (v)$ is 
bounded in $L^2_tL^2_xH^{1/2}_y$, uniformly in $\e$. 
However, according to Lemma \ref{lemma:energy}, 
the time derivative $\dt v^\e$ is merely of order 
$1/\e$ in the $L^2_tH^{-1/2}_x$ norm. Therefore, we cannot expect 
$\dt \Tau_{\e}^b(v^\e)$ (assuming that this object is meaningful) 
to be bounded in $L^2_tL^2_xH^{-1/2}_y$, 
uniformly in $\e$. As a consequence, strong $L^2_{t,y}$ compactness of 
$\seq{\Tau_{\e}^b (v^\e)}_{\e>0}$ is not deducible 
from the classical Aubin-Lions theorem. Instead of attempting 
to control (in a negative space) the whole derivative $\dt \Tau_{\e}^b(v^\e)$, we 
will make use of a temporal translation estimate with respect to 
the $L^2$ norm (cf.~lemma below). The $L^2_{t,y}$ compactness will then 
be a consequence of the Simon lemma (cf.~Subsection \ref{sec:notation}).

The rest of this section is devoted to the detailed proof 
that $\seq{\Tau_{\e}^b(v^{\e})}_{\e>0}$ is strongly 
precompact in the fixed space $L^2((0,T)\times \O\times\G)$. 
We begin with
\begin{lemma}\label{lem:Tb-temporal-spatial}
There exists a constant $C$, independent of $\e$, such that 
\begin{equation}\label{eq:TbH12}
	\norm{\Tau_{\e}^b (v^{\e})}_{L^2(0,T;L^2(\O;H^{1/2}(\G))}\le C
\end{equation}
and 
\begin{equation}\label{eq:Tb-temporal}
	\norm{\Tau_{\e}^b(v^\e)(\cdot+\Delta_t,\cdot,\cdot)
	-\Tau_{\e}^b(v^\e)(\cdot,\cdot,\cdot)}_{L^2(0,T-\Delta_t;L^2(\O\times \G))}
	\le C \Delta_t^{\frac12},
\end{equation}
for sufficiently small temporal shifts $\Delta_t>0$.
\end{lemma}

\begin{proof}
By \eqref{ve-def} and the linearity of $\Tau_\e^b(\cdot)$, we have 
$\Tau_{\e}^b (v^{\e})=\Tau_{\e}^b(u_i^{\e}|_{\Ge}) 
- \Tau_{\e}^b(u_e^{\e}|_{\Ge})$. Thus, \eqref{eq:TbH12} 
follows by integrating \eqref{unfoldingH1/2} 
over $(0,T)$ and using estimates (a), (b) in \eqref{eq:main-est}. 
Regarding \eqref{eq:Tb-temporal}, we use the linearity of 
$\Tau_{\e}^b$, \eqref{Tb-L2Bound}, and \eqref{eq:ve-time-trans} to obtain
\begin{align*}
	& \int_0^{T-\Delta_t} \int_\O \int_\G
	\abs{\Tau_{\e}^b(v^\e)(t+\Delta_t,x,y)
	-\Tau_{\e}^b(v^\e)(\cdot,\cdot,\cdot)}^2\, dS(y) \,dx\,dt
	\\ & \quad 
	= \e \int_0^{T-\Delta_t}\int_{\Ge} 
	\abs{v^\e(t+\Delta_t,x)-v^\e(t,x)}^2 \,dS(x) \,dt
	\le C \Delta_t.
\end{align*}
\end{proof}

Let us think of $\Tau_{\e}^b(v^\e)$ as a function 
of $x\in \O$, with values in $L^2(0,T;L^2(\G))$. 
Fixing $x$, in view of Lemma \ref{lem:Tb-temporal-spatial} 
and Simon's compactness criterion, the sequence 
$\seq{(t,y)\mapsto \Tau_{\e}^b(v^{\e})(t,x,y)}_{\e>0}$ is 
precompact in $L^2((0,T)\times \G)$. The $x$-variable is more difficult. 
As a matter of fact, since $\Tau_{\e}^b(v^{\e})$ is piecewise 
constant as a function of $x$ and thus does not belong to any Sobolev 
space, strong compactness in $x$ is not immediately clear. 
We address this issue by deriving a translation 
estimate in the $x$-variable. To be more precise, we make use of a convenient 
Simon-type compactness criterion ($x$ playing the role of time) established in 
\cite[Corollary 2.5]{Neuss} (see also \cite[Section 5]{Amann:2000aa}), 
which is recalled next.  

For $Q\subset \R^n$ and $\xi\in \R^n$ ($n\ge 1$), we 
set $Q_\xi:=Q \cap (Q-\xi):=\seq{x\in Q: x+\xi\in Q}$, 
$\Sigma:=\seq{-1,1}^n$ ($\abs{\Sigma}=2^n$), and 
$\xi_\sigma:=(\xi_1 \sigma_1,\ldots,\xi_n\sigma_n)\in \R^n$ 
for $\sigma\in \Sigma$. If $Q=(a,b)$ ($:=\Pi_{\ell=1}^n (a_\ell,b_\ell)$ for 
$a,b\in \R^n$, $a<b$) is an open rectangle, then 
\begin{equation}\label{Qxi-decomp}
	Q=\bigcup_{\sigma\in \Sigma} Q_{\xi_\sigma}, \qquad 
	\text{for any $\xi \in \R^n$ such that $Q_\xi \neq \emptyset$}.
\end{equation}
Let $B$ be Banach a space. 
For $f:Q\to B$ and $\Delta\in \R^n$, we define the 
translation operator $\tau_\Delta:(Q-\Delta)\to B$ by
$$
\tau_\Delta f(x)=f(x+\Delta).
$$
The following theorem, due to Gahn and Neuss-Radu \cite{Neuss}, 
is a multi-dimensional generalization of Simon's main 
result \cite[Theorem 1]{Simon:1987vn}.

\begin{theorem}[\cite{Neuss}]\label{compactSurface0}
Let $\mathcal{F} \subset L^p(Q;B)$ for some 
Banach space $B$, open rectangle $Q=(a,b)\subset \R^n$, and 
$p\in [1,\infty)$. Then $\mathcal{F}$ is 
precompact in $L^p(Q;B)$ if and only if

\begin{enumerate}[i.]
	\item $\seq{\int_A f \,dx \, \big | \, f \in \mathcal{F}}$ 
	is precompact in $B$, for every open rectangle $A\subset Q$;

	\item for each $z\in \R^n$ with $0<z<b-a$,
	\begin{equation}\label{shift-var-domain}
		\sup_{f \in \mathcal{F}} 
		\norm{\tau_z f - f }_{L^p(Q_z;B)}
		\to 0, \qquad \text{as $z \to 0$.}
	\end{equation}
\end{enumerate}
\end{theorem}

Recall \eqref{Qxi-decomp}, this time specifying $Q=(a,b)\subset \R^n$, 
and $\xi=(b-a)/2\in \R^n$. Condition \eqref{shift-var-domain} 
in Theorem \ref{compactSurface0} is equivalent to \cite{Neuss}
\begin{equation}\label{shift-fixed-domain}
	\sup_{f \in \mathcal{F}} 
	\norm{\tau_{z_\sigma} f - f }_{L^p(Q_{\xi_\sigma};B)}
	\overset{z\to 0}{\to} 0, \qquad \text{$z\in \R^n$, $z\ge 0$, 
	$\forall \sigma\in \Sigma$.}		
\end{equation}
The difference between \eqref{shift-var-domain} 
and \eqref{shift-fixed-domain} is the 
fixed domain that is utilized in the latter 
(it does not depend on the shift $z$). 
We make use of \eqref{shift-fixed-domain} in the proof 
of Lemma \ref{lem:ver-ii} below.

We now verify that the sequence 
$\seq{\Tau_{\e}^b(v^{\e})}_{\e>0}$ of 
unfolded membrane potentials satisfies the assumptions 
of Theorem \ref{compactSurface0}, with 
$B= L^2((0,T)\times \G)$, $Q = \O$, $p=2$.

\begin{lemma}[verification of $i$]\label{lem:ver-i}
Given an arbitrary open rectangle $A\subset \O$, define 
the function $v^{\e}_A(t,y)$ by
$$
v^{\e}_A(t,y) 
=\int_A \Tau_{\e}^b(v^\e)(t,x,y)\,dx, 
\qquad (t,x)\in (0,T)\times \G.
$$
Then the sequence $\seq{v^\e_A}_{\e>0}$ is 
precompact in $L^2((0,T)\times \G)$.  
\end{lemma}

\begin{proof}
In view of Jensen's inequality, it follows that
\begin{align*}
	\norm{v_A^{\e}}_{L^2(0,T;H^{1/2}(\G))}^2 
	& = \int_0^T
	\norm{\int_A \Tau_{\e}^b(v^{\e})(t,x,\cdot) \,dx}_{H^{1/2}(\G)}^2 \, dt 
	\\ & \leq \abs{A}\int_0^T\int_A
	\norm{\Tau_{\e}^b (v^{\e})(t,x,\cdot)}^2_{H^{1/2}(\G)} \,dx\, dt 
	\\ & \le \abs{A}\norm{\Tau_{\e}^b (v^\e)}_{L^2(0,T;L^2(\O;H^{1/2}(\Ge))}^2
	\overset{\eqref{eq:TbH12}}{\le} C.
\end{align*}
Let $\Delta_t>0$ be a small temporal shift. 
Again using Jensen's inequality,
\begin{align*}
	& \norm{v_A^{\e}(\cdot+\Delta_t,\cdot)
	-v_A^{\e}(\cdot,\cdot}_{L^2(0,T-\Delta_T;L^2(\G))}^2
	\\ & \quad
	\leq \abs{A}\int_0^{T-\Delta_t}\int_A 
	\norm{\Tau_{\e}^b (v^{\e})(t+\Delta_t,x,\cdot)-
	\Tau_{\e}^b (v^{\e})(t,x,\cdot)}^2_{L^2(\G)} \,dx\, dt 
	\\ & \quad
	\le \abs{A}\norm{\Tau_{\e}^b (v^{\e})(\cdot+\Delta_t,\cdot,\cdot)-
	\Tau_{\e}^b (v^{\e})(\cdot,\cdot,\cdot)}^2_{L^2(0,T-\Delta_t,L^2(\O\times\G))}
	\overset{\eqref{eq:Tb-temporal}}{\le} C \Delta_t.
\end{align*}

Summarizing, there exists an $\e$-independent constant $C$ such that
$$
\norm{v_A^{\e}}_{L^2(0,T;H^{1/2}(\G))}\le C, \quad
\norm{v_A^{\e}(\cdot+\Delta_t,\cdot)
-v_A^{\e}(\cdot,\cdot}_{L^2(0,T-\Delta_T;L^2(\G))}
\le C \Delta_t^{1/2}.
$$
The lemma follows from these estimates and Simon's compactness criterion.
\end{proof}

In the next lemma we verify \eqref{shift-fixed-domain} with 
$\xi=(1/2,1/2,1/2)\in \R^3$, which is equivalent to 
condition $ii$ in Theorem \ref{compactSurface0}. 

\begin{lemma}[verification of $ii$]\label{lem:ver-ii}
Given any $\delta>0$, there exists $h>0$ such that 
for any $\Delta_x\in \R^3$ with $\abs{\Delta_x}<h$ 
and for all $\e\in (0,1]$,
\begin{equation}\label{eq:Tb-xtranslate}
	\norm{\Tau_{\e}^b(v^\e)(\cdot,\cdot+(\Delta_x)_\sigma,\cdot)
	-\Tau_{\e}^b(v^\e)(\cdot,\cdot,\cdot)}_{L^2(\O_{\xi_\sigma};
	L^2((0,T)\times \G))} < \delta, 
	\quad \forall \sigma\in \Sigma.
\end{equation}
\end{lemma}

\begin{proof}
We wish to estimate the quantity 
\begin{align*}
	& J(\Delta_x;\e)
	:= \norm{\Tau_{\e}^b(v^\e)(\cdot,\cdot+(\Delta_x)_\sigma,\cdot)
	-\Tau_{\e}^b(v^\e)(\cdot,\cdot,\cdot)}_{L^2(\O_{\xi_\sigma};
	L^2((0,T)\times \G))}^2
	\\ & \quad = \int_0^T\!\!\int_{\O_{\xi_\sigma}}\int_\G
	\abs{v^\e\left(t,\e\left\lfloor \frac{x+(\Delta_x)_\sigma}{\e}
	\right\rfloor +\e y\right)
	-v^\e\left(t,\e\left\lfloor \frac{x}{\e}
	\right\rfloor +\e y\right)}^2\, dS(y)\, dx\, dt.
\end{align*}
Recall that $\e$ takes values in a sequence $\subset (0,1]$ 
tending to zero. Fix any $\e_0>0$. Since the translation 
operation is continuous in $L^2$, there exists $h_0=h_0(\e_0)>0$ 
such that $J(\Delta_x;\e)<\delta$ for any $\abs{\Delta_x}<h_0$, 
for all $\e\in [\e_0,1]$. The rest of the proof is devoted 
to arguing that this holds also for $\e\in (0,\e_0)$, 
thereby proving \eqref{eq:Tb-xtranslate}.
 
Choose $K^\e_{\xi_\sigma}\subset \Z^3$ such that
$\O_{\xi_\sigma}
=\text{interior}\left(\bigcup_{k\in K^\e_{\xi_\sigma}} 
\overline{\e Y^k}\right), \qquad Y^k := k+Y$. 
Then $J(\Delta_x,\e)$ becomes
$$
\sum_{k\in K^\e_{\xi_\sigma}} 
\int_0^T\!\!\int_{\e Y^k}\int_\G
\abs{v^\e\left(t,\e\left\lfloor \frac{x+(\Delta_x)_\sigma}{\e}
\right\rfloor +\e y\right)
-v^\e\left(t,\e\left\lfloor \frac{x}{\e}
	\right\rfloor +\e y\right)}^2	\, dS(y)\, dx\, dt.
$$
If $x\in \e Y^k$, then $\left\lfloor \frac{x}{\e}\right\rfloor=k$, 
but we have no useful information about 
$\left\lfloor \frac{x+(\Delta_x)_\sigma}{\e}\right\rfloor$.
To address this issue, we make use of a favorable 
decomposition of the cells $\e Y^k$ proposed 
in \cite{Neuss-Radu:2007aa} (and also utilized 
in e.g.~\cite{Neuss,Gahn:2016aa}). 

We decompose each cell $\e Y^k$ as
$$
\e Y^k=\bigcup_{m\in \seq{0,1}^3}\e Y^{k,m}_\sigma,
\quad
\e Y^{k,m}_\sigma := \seq{x\in \e Y^k:
\e\left\lfloor \frac{x+\e\seq{\frac{(\Delta_x)_\sigma}{\e}}}{\e}\right\rfloor
=\e (k+m_\sigma)},
$$
for $k\in K^\e_{\xi_\sigma}$ and $\sigma\in\Sigma$. 
Regarding the translation, for $x\in \e Y^{k,m}_\sigma$, we write
$(\Delta_x)_\sigma=\e\left\lfloor \frac{(\Delta_x)_\sigma}{\e}\right\rfloor
+\e \seq{\frac{(\Delta_x)_\sigma}{\e}}$, and note that
\begin{align*}
	\e\left\lfloor \frac{x+(\Delta_x)_\sigma}{\e}\right\rfloor
	& =\e\left\lfloor \frac{x+\e \seq{\frac{(\Delta_x)_\sigma}{\e}}}{\e}
	+\left\lfloor \frac{(\Delta_x)_\sigma}{\e}\right\rfloor\right\rfloor
	\\ &
	=\e\left\lfloor \frac{x+\e \seq{\frac{(\Delta_x)_\sigma}{\e}}}{\e}\right\rfloor
	+\e\left\lfloor \frac{(\Delta_x)_\sigma}{\e}\right\rfloor
	\\ & 
	= \e (k+m_\sigma)
	+\e\left\lfloor \frac{(\Delta_x)_\sigma}{\e}\right\rfloor.
\end{align*}
As a result of this,
\begin{align*}
	J & =\sum_{k\in K^\e_{\xi_\sigma}}
	\sum_{m\in \seq{0,1}^3} 
	\int_0^T \int_{\e Y^{k,m}_\sigma}\int_\G
	\\ & \qquad\qquad 
	\times \abs{v^\e\left(t,\e k+\e m_\sigma
	+\e\left\lfloor \frac{(\Delta_x)_\sigma}{\e}\right\rfloor +\e y\right)
	-v^\e(t,\e k +\e y)}^2 \, dS(y)\, dx\, dt
	\\ & 
	\overset{\left(\overset{x:=\e k + \e y}{dS(x)=\e^2 dS(y)}\right)}{\le}
	\e \sum_{k\in K^\e_{\xi_\sigma}}
	\sum_{m\in \seq{0,1}^3}\int_0^T\int_{\e(k+\G)}
	\\ & \qquad\qquad\qquad\qquad\qquad 
	\times \abs{v^\e\left(t,x+\e\left(m_\sigma
	+\left\lfloor \frac{(\Delta_x)_\sigma}{\e}\right\rfloor\right)\right)
	-v^\e(t,x)}^2 \, dS(x) \, dt,
\end{align*}
where we have also used $\int_{\e Y^{k,m}_\sigma}\, dx 
\le \int_{\e Y^k} \,dx=\e^3$ to 
arrive at the final line. Since 
$\sum_{k\in K^\e_{\xi_\sigma}}\int_{\e(k+\G)}
=\int_{(\Ge)_{\xi_\sigma}}$, we conclude that
$$
J\le \e \sum_{m\in \seq{0,1}^3}\int_0^T\int_{(\Ge)_{\xi_\sigma}}
\abs{v^\e(t,x+z)-v^\e(t,x}^2 \, dS(x) \, dt,
$$
where the shift $z=z(\Delta_x,\e,m)$ is $\e\left(m_\sigma
+\left\lfloor \frac{(\Delta_x)_\sigma}{\e}\right\rfloor\right)$, i.e., 
$z$ is an integer-multiple of $\e$. Note that $x+z\in (\Ge)_{\xi_\sigma}$ whenever  
$\Delta_x$ and $\e$ are sufficiently small. 
Recalling the definition \eqref{ve-def} of $v^\e$, 
the trace inequality \eqref{trace} implies
\begin{align*}
	&\e \int_0^T\int_{(\Ge)_{\xi_\sigma}}
	\abs{v^\e(t,x+z)-v^\e(t,x}^2 \, dS(x) \, dt 
	\\ & \qquad 
	\leq C \sum_{j=i,e} 
	\int_0^T \int_{(\Oj)_{\xi_\sigma}}
	\abs{u_j^\e(t,x+z) -u_j^\e(t,x)}^2 \,dx \, dt 
	\\ & \qquad \qquad 
	+ C \e^2 \sum_{j=i,e} 
	\int_0^T\int_{(\Oj){\xi_\sigma}}
	\abs{\nabla u_j^\e(t,x+z)-\nabla u_j^\e(t,x)}^2 \,dx\, dt,
\end{align*}
where the last term is bounded by a constant times $\e^2$ 
because of \eqref{eq:main-est}-(a). 

It remains to estimate the term on the second line, 
which will be done utilizing the well-known characterization of 
Sobolev spaces by means of translation (difference) operators. 
Recalling the standard proof of this characterization, a problem 
that arises (due to the geometry of $\Oj$) is that parts of the 
line segment between $x$ and $z$ may leave $\Oj$. 
To avoid this problem we make use of the interpolation 
operators \eqref{Qdefinition} to obtain functions 
$Q_\e^j(u_j^\e)$ defined on the whole of $\Omega$. 	

Using the triangle inequality and 
recalling the estimates \eqref{Qestimates}, we obtain
\begin{align*}
	&\int_0^T \int_{(\Oj)_{\xi_\sigma}} 
	\abs{u_j^\e(t,x+z) -u_j^\e(t,x)}^2 \,dx \, dt 
	\\ & \quad \leq \int_0^T \int_{(\Oj)_{\xi_\sigma}}
	\abs{u_j^\e(t,x+z)-Q_\e^j(u_j^\e)(t,x+z)}^2 \,dx \, dt 
	\\ & \quad\qquad\qquad
	+\int_0^T \int_{(\Oj)_{\xi_\sigma}}
	\abs{Q_\e^j(u_j^\e)(t,x+z)-Q_\e^j(u_j^\e)(t,x)}^2 \,dx \, dt 
	\\ & \quad\qquad\quad\qquad
	+\int_0^T \int_{(\Oj)_{\xi_\sigma}}
	\abs{Q_\e^j(u_j^\e)(t,x)-u_j^\e(t,x)}^2 \,dx\, dt 
	\\ & \quad
	\leq C_1 \e \int_0^T \int_{(\Oj)_{\xi_\sigma}}\abs{\nabla u_j^\e}^2 \,dx \, dt  
	+ C_2 \abs{z}\int_0^T 
	\int_{(\Oj)_{\xi_\sigma}}\abs{\nabla Q_\e^j(u_j^\e)}^2 \,dx \, dt 
	\\ & \quad 
	\leq C_3\bigl( \e + \abs{z}\bigr)
	\norm{\nabla u_j^\e}^2_{L^2(0,T;L^2(\Oj))}
	\overset{\eqref{eq:main-est}-(a)}{\le} 
	C_4 \bigl( \e + \abs{z}\bigr).
\end{align*}

Hence, 
$$
J \le C_5 \e + C_6 \abs{\Delta_x},
$$
where the constants $C_5,C_6$ are independent of $\Delta_x,\e$.
We select the $\e_0$ introduced earlier sufficiently small, such 
that the first term on the right-hand side is $<\delta^2/2$ for 
all $\e<\e_0$. We pick $h_1>0$ such that the second term is 
$<\delta^2/2$ for all $\abs{\Delta_x}<h_1$ (for any $\e\in (0,1]$). 
Specifying $h:=\min(h_0,h_1)$, the 
claim \eqref{eq:Tb-xtranslate} now follows.
\end{proof}

\subsection{Concluding the proof of Theorem \ref{theorem:homo}}\label{sec:concl}
Summarizing, we know that 
$$
\Tau_{\e}^b(w^{\e})\overset{\e\downarrow 0}{\weak} w 
\quad \text{in $L^2((0,T)\times \O\times\G)$}, 
$$
because $w^{\e}\overset{2-\mathrm{S}}{\rightharpoonup} w$, 
cf.~\eqref{convergence} and \eqref{weak-vs-twoscale}. 
Besides, $w^{\e}$ appears linearly in $I$ and $H$, cf.~\eqref{GFHN}. 
We have shown that $\seq{\Tau_{\e}^b(v^{\e})}_{\e>0}$ 
is strongly precompact in $L^2((0,T)\times \O\times\G)$. 
It then follows that $\Tau_{\e}^b(v^{\e})\to v$ 
in $L^2((0,T)\times \O\times\G)$ and a.e.~in $(0,T)\times \O\times\G$, 
along a subsequence as $\e\to 0$ (not relabelled), where $v$ is 
the two-scale limit of $\seq{v^{\e}}_{\e>0}$ 
identified in \eqref{convergence}. By way of estimate (d) in 
\eqref{eq:main-est} and (the $L^p$ version of) \eqref{Tb-L2Bound}, 
$$
\norm{\Tau_{\e}^b (v^{\e})}_{L^4((0,T)\times\O\times \G)}
= \e^{1/4}\norm{v^{\e}}_{L^4((0,T)\times\Ge)}\le C,
$$
where $C$ is independent of $\e$. In view of this estimate and 
the Vitali convergence theorem, we conclude 
the validity of \eqref{limitIntegrals}. This finishes the 
proof of Theorem \ref{theorem:homo}.  

\bibliographystyle{abbrv}

\begin{thebibliography}{10}

\bibitem{Allaire}
G.~Allaire.
\newblock Homogenization and two-scale convergence.
\newblock {\em SIAM J. Math. Anal.}, 23(6):1482--1518, 1992.

\bibitem{Allaire2}
G.~Allaire, A.~Damlamian, and U.~Hornung.
\newblock Two-scale convergence on periodic surfaces and applications.
\newblock In {A. Bourgeat et al.}, editor, {\em Proceedings of the
  International Conference on Mathematical Modelling of Flow through Porous
  Media (May 1995)}, pages 15--25. World Scientific Pub., Singapore, 1996.

\bibitem{Amann:2000aa}
H.~Amann.
\newblock Compact embeddings of vector-valued {S}obolev and {B}esov spaces.
\newblock {\em Glas. Mat. Ser. III}, 35(55)(1):161--177, 2000.

\bibitem{Amar2013}
M.~Amar, D.~Andreucci, P.~Bisegna, and R.~Gianni.
\newblock A hierarchy of models for the electrical conduction in biological
  tissues via two-scale convergence: the nonlinear case.
\newblock {\em Differential Integral Equations}, 26(9-10):885--912, 2013.

\bibitem{Andreianov:2010uq}
B.~Andreianov, M.~Bendahmane, K.~H. Karlsen, and C.~Pierre.
\newblock Convergence of discrete duality finite volume schemes for the cardiac
  bidomain model.
\newblock {\em Netw. Heterog. Media}, 6(2):195--240, 2011.

\bibitem{Karlsen}
M.~Bendahmane and K.~H. Karlsen.
\newblock Analysis of a class of degenerate reaction-diffusion systems and the
  bidomain model of cardiac tissue.
\newblock {\em Netw. Heterog. Media}, 1(1):185--218, 2006.

\bibitem{Boulakia2008}
M.~Boulakia, M.~A. Fern\'andez, J.-F. Gerbeau, and N.~Zemzemi.
\newblock A coupled system of {PDE}s and {ODE}s arising in electrocardiograms
  modeling.
\newblock {\em Appl. Math. Res. Express. AMRX}, (2):Art. ID abn002, 28, 2008.

\bibitem{Bourgault}
Y.~Bourgault, Y.~Coudi{\`e}re, and C.~Pierre.
\newblock {Existence and uniqueness of the solution for the bidomain model used
  in cardiac electrophysiology}.
\newblock {\em {Nonlinear Analysis: Real World Applications}}, 10(1):458--482,
  2009.

\bibitem{Boyer:2012aa}
F.~Boyer and P.~Fabrie.
\newblock {\em Mathematical Tools for the Study of the Incompressible
  Navier-Stokes Equations and Related Models}.
\newblock Applied Mathematical Sciences. Springer New York, 2012.

\bibitem{Cioranescu:2008aa}
D.~Cioranescu, A.~Damlamian, and G.~Griso.
\newblock The periodic unfolding method in homogenization.
\newblock {\em SIAM J. Math. Anal.}, 40(4):1585--1620, 2008.

\bibitem{Cioranescu}
D.~Cioranescu, A.~Damlamian, P.~Donato, G.~Griso, and R.~Zaki.
\newblock The periodic unfolding method in domains with holes.
\newblock {\em SIAM J. Math. Anal.}, 44(2):718--760, 2012.

\bibitem{Donato}
D.~Cioranescu and P.~Donato.
\newblock {\em An introduction to homogenization}, volume~17 of {\em Oxford
  Lecture Series in Mathematics and its Applications}.
\newblock The Clarendon Press, Oxford University Press, New York, 1999.

\bibitem{ColliBook}
P.~Colli~Franzone, L.~F. Pavarino, and S.~Scacchi.
\newblock {\em Mathematical cardiac electrophysiology}, volume~13 of {\em
  MS\&A. Modeling, Simulation and Applications}.
\newblock Springer, Cham, 2014.

\bibitem{Colli}
P.~Colli~Franzone and G.~Savar\'e.
\newblock Degenerate evolution systems modeling the cardiac electric field at
  micro- and macroscopic level.
\newblock In {\em Evolution equations, semigroups and functional analysis
  ({M}ilano, 2000)}, volume~50 of {\em Progr. Nonlinear Differential Equations
  Appl.}, pages 49--78. Birkh\"auser, Basel, 2002.

\bibitem{Donato:2015aa}
P.~Donato and K.~H. Le~Nguyen.
\newblock Homogenization of diffusion problems with a nonlinear interfacial
  resistance.
\newblock {\em NoDEA Nonlinear Differential Equations Appl.}, 22(5):1345--1380,
  2015.
  
\bibitem{Donato:2011aa}
P.~Donato, K.~H. Le~Nguyen, and R.~Tardieu.
\newblock The periodic unfolding method for a class of imperfect transmission
  problems.
\newblock {\em J. Math. Sci. (N.Y.)}, 176(6):891--927, 2011.
\newblock Problems in mathematical analysis. No. 58.

\bibitem{FitzHugh1955}
R.~FitzHugh.
\newblock Mathematical models of threshold phenomena in the nerve membrane.
\newblock {\em The bulletin of mathematical biophysics}, 17(4):257--278, Dec
  1955.
  
\bibitem{Gahn:2016aa}
M.~Gahn, M.~Neuss-Radu, and P.~Knabner.
\newblock Homogenization of reaction--diffusion processes in a two-component
  porous medium with nonlinear flux conditions at the interface.
\newblock {\em SIAM Journal on Applied Mathematics}, 76(5):1819--1843, 2016.

\bibitem{Neuss}
M.~Gahn and M.~Neuss-Radu.
\newblock A characterization of relatively compact sets in {$L^p(\Omega, B)$}.
\newblock {\em Stud. Univ. Babe\c s-Bolyai Math.}, 61(3):279--290, 2016.

\bibitem{Graf:2014aa}
I.~Graf and M.~A. Peter.
\newblock Diffusion on surfaces and the boundary periodic unfolding operator
  with an application to carcinogenesis in human cells.
\newblock {\em SIAM J. Math. Anal.}, 46(4):3025--3049, 2014.

\bibitem{Grandelius:2017aa}
E.~Grandelius.
\newblock The bidomain equations of cardiac electrophysiology.
\newblock Master's thesis, University of Oslo, 2017.

\bibitem{Henriquez}
C.~S. Henriquez and W.~Ying.
\newblock {\em The Bidomain Model of Cardiac Tissue: From Microscale to Macroscale}.
\newblock Springer US, Boston, MA, 2009.

\bibitem{Hodgkin}
A.~L. Hodgkin and A.~F. Huxley.
\newblock A quantitative description of membrane current and its application to
  conduction and excitation in nerve.
\newblock {\em J. Physiol.}, 117(4):500--544, 1952.

\bibitem{Hornung:1991aa}
U.~Hornung and W.~J\"ager.
\newblock Diffusion, convection, adsorption, and reaction of chemicals in
  porous media.
\newblock {\em J. Differential Equations}, 92(2):199--225, 1991.

\bibitem{Keener:1996aa}
J.~P. Keener and A.~V. Panfilov.
\newblock A biophysical model for defibrillation of cardiac tissue.
\newblock {\em Biophysical Journal}, 71(3):1335--1345, 1996.

\bibitem{Keener:1998ab}
J.~P. Keener.
\newblock The effect of gap junctional distribution on defibrillation.
\newblock {\em Chaos}, 8(1):175--187, 1998.

\bibitem{Lions:1972aa}
J.~Lions and E.~Magenes.
\newblock {\em Non-homogeneous boundary value problems and applications}.
\newblock Number v. 3 in Non-homogeneous Boundary Value Problems and
  Applications. Springer-Verlag, 1972.

\bibitem{Lukkassen:2002aa}
D.~{Lukkassen}, G.~{Nguetseng}, and P.~{Wall}.
\newblock {Two-scale convergence.}
\newblock {\em {Int. J. Pure Appl. Math.}}, 2(1):35--86, 2002.

\bibitem{Marciniak-Czochra:2008aa}
A.~Marciniak-Czochra and M.~Ptashnyk.
\newblock Derivation of a macroscopic receptor-based model using homogenization
  techniques.
\newblock {\em SIAM J. Math. Anal.}, 40(1):215--237, 2008.

\bibitem{McLean:2000aa}
W.~McLean.
\newblock {\em Strongly elliptic systems and boundary integral equations}.
\newblock Cambridge University Press, 2000.

\bibitem{Neu:1993aa}
J.~C. Neu and W.~Krassowska.
\newblock Homogenization of syncytial tissues.
\newblock {\em Crit. Rev. Biomed. Eng.}, 21(2):137--199, 1993.

\bibitem{Neuss-Radu:2007aa}
M.~Neuss-Radu and W.~J\"ager.
\newblock Effective transmission conditions for reaction-diffusion processes in
  domains separated by an interface.
\newblock {\em SIAM J. Math. Anal.}, 39(3):687--720, 2007.

\bibitem{Nguetseng}
G.~Nguetseng.
\newblock A general convergence result for a functional related to the theory
  of homogenization.
\newblock {\em SIAM J. Math. Anal.}, 20(3):608--623, May 1989.

\bibitem{Pennacchio2005}
M.~Pennacchio, G.~Savar\'e, and P.~Colli~Franzone.
\newblock Multiscale modeling for the bioelectric activity of the heart.
\newblock {\em SIAM J. Math. Anal.}, 37(4):1333--1370, 2005.

\bibitem{Richardson:2009aa}
G.~Richardson.
\newblock A multiscale approach to modelling electrochemical processes
  occurring across the cell membrane with application to transmission of action
  potentials.
\newblock {\em Mathematical Medicine and Biology: A Journal of the IMA},
  26(3):201--224, 2009.

\bibitem{Richardson:2011aa}
G.~Richardson and S.~J. Chapman.
\newblock Derivation of the bidomain equations for a beating heart with a
  general microstructure.
\newblock {\em SIAM J. Appl. Math.}, 71(3):657--675, 2011.

\bibitem{Simon:1987vn}
J.~Simon.
\newblock Compact sets in the space {$L\sp p(0,T;B)$}.
\newblock {\em Ann. Mat. Pura Appl. (4)}, 146:65--96, 1987.

\bibitem{Sundnes}
J.~Sundnes, G.~T. Lines, X.~Cai, B.~r.~F. Nielsen, K.-A. Mardal, and A.~Tveito.
\newblock {\em Computing the electrical activity in the heart}, volume~1 of
  {\em Monographs in Computational Science and Engineering}.
\newblock Springer-Verlag, Berlin, 2006.

\bibitem{Tung78}
L.~Tung.
\newblock {\em A bi-domain model for describing ischemic myocardial D-C
  potentials}.
\newblock PhD thesis, MIT, Cambridge, MA, 1978.

\bibitem{Tveito17}
A.~Tveito, K.~H. J{\ae}ger, M.~Kuchta, K.-A. Mardal, and M.~E. Rognes.
\newblock A cell-based framework for numerical modeling of electrical
  conduction in cardiac tissue.
\newblock {\em Frontiers in Physics}, 5:48, 2017.

\bibitem{Veneroni}
M.~Veneroni.
\newblock Reaction-diffusion systems for the microscopic cellular model of the
  cardiac electric field.
\newblock {\em Math. Methods Appl. Sci.}, 29(14):1631--1661, 2006.

\bibitem{Veneroni:2009aa}
M.~Veneroni.
\newblock Reaction-diffusion systems for the macroscopic bidomain model of the
  cardiac electric field.
\newblock {\em Nonlinear Anal. Real World Appl.}, 10(2):849--868, 2009.

\bibitem{Yang:2014aa}
Z.~Yang.
\newblock The periodic unfolding method for a class of parabolic problems with
  imperfect interfaces.
\newblock {\em ESAIM Math. Model. Numer. Anal.}, 48(5):1279--1302, 2014.

\end{thebibliography}

\end{document}